\documentclass[11pt]{article}

\usepackage{amstext,amssymb,amsmath,amsbsy}
\usepackage{hyperref}
\usepackage{amscd}
\usepackage{amsfonts}
\usepackage{indentfirst}
\usepackage{verbatim}
\usepackage{amsmath}
\usepackage{amsthm}
\usepackage{enumerate}
\usepackage{graphicx}
\usepackage{color}
\usepackage[OT1]{fontenc}
\usepackage[latin1]{inputenc}
\usepackage[english]{babel}
\usepackage{amssymb}
\usepackage{subfig}
\usepackage{algorithm}
\usepackage{algpseudocode}

\newcommand{\R}{\mathbb{R}}

\newcommand{\Id}{\textrm{Id}}

\newcommand{\bx}{{\bf x}}

\newtheorem{Theorem}{Theorem}[section]

\newtheorem{Proposition}{Proposition}[section]

\newtheorem{remark}{Remark}[section]

\newtheorem*{Assumption*}{Assumption}

\newtheorem{problem}{Problem}[section]
\newtheorem*{problem*}{Problem}
\numberwithin{equation}{section}

\begin{document}
	
\title{A numerical method for an inverse source problem for parabolic equations and its application to a coefficient inverse problem }

\author{Phuong Mai Nguyen\thanks{Department of Mathematics and Statistics, University of North Carolina Charlotte, Charlotte, NC, 28223, USA, \text{pnguye45@uncc.edu}}
\and
Loc Hoang Nguyen\thanks{Department of Mathematics and Statistics, University of North Carolina Charlotte, Charlotte, NC, 28223, USA, \text{loc.nguyen@uncc.edu}, corresponding author}
}

\date{}
\maketitle

\begin{abstract}
	Two main aims of this paper are to  develop a numerical method to solve an inverse source problem for parabolic equations and  apply it to solve a nonlinear coefficient inverse problem. 	
	The inverse source problem in this paper is the problem to reconstruct a source term from external observations.
	Our method to solve this inverse source problem consists of two stages. We first establish an equation of the derivative of the solution to the parabolic equation with respect to the time variable.  
	Then, in the second stage, we solve this equation by the quasi-reversibility method.
	The inverse source problem considered in this paper is the linearization of a nonlinear coefficient inverse problem. 
	Hence, iteratively solving the inverse source problem provides the numerical solution to that coefficient inverse problem.
	Numerical results for the inverse source problem under consideration and the corresponding nonlinear coefficient  inverse problem are presented.
\end{abstract}

\noindent \textit{Key words.} 
parabolic equation, 
inverse source problem,
coefficient inverse problem,
numerical method,
 quasi-reversibility method

\noindent \textit{AMS Classification} 35R30, 35K20
	
\section{Introduction}	\label{sec Intro}

The area of inverse source problems has many applications and it, therefore, attracts the attention of the scientific community, see e.g., \cite{BadiaDuong:jiip2002,Devaney:itsu1983,HusseinLesnic:ejbe2014,
HusseinLesnic:jem2016,HusseinLesnic:jem2016t,
LesnicHusseinJohabsson:jcam2016, Malyshev:jmaa1989,
LocNguyen:ip2019, NguyenLiKlibanov:IPI2019}. The  solutions of inverse source problems can be used to directly detect the source even when the source is inactive after a certain time. Here, we name some examples. In the case of the parabolic equation, the problem plays an important role in identifying the pollution sources in a river or a lake \cite{BadiaDuong:jiip2002}. In the case of elliptic equations, the inverse source problem has applications in electroencephalography \cite{Anastasioelal:ip2007, Devaney:itsu1983}. In the case that the data are generated by an
acoustic source, the governing equation is the hyperbolic one and the
problem addresses ultrasonics imaging and photoacoustic tomography \cite%
{Anastasioelal:ip2007, Devaney:itsu1983}.
In this paper, we propose a numerical method to solve an inverse source problem for parabolic equations. 
This problem is the linearization of a nonlinear coefficient inverse problem. 
Therefore, we can use it  to solve a coefficient inverse problem.

Let $\Omega$ be a bounded domain in $\R^d$, $d \geq 1$, with smooth boundary $\partial \Omega$.
Let  $c$ be a function in the class $C^1(\Omega)$. 
Consider the function $u = u(\bx, t)$, $\bx \in \Omega$, $t > 0$ that is governed by the following initial value problem
\begin{equation}
	\left\{
		\begin{array}{rcll}
			u_{t}(\bx, t) &=&  \mathcal{A} u(\bx, t)+ f(\bx, t)p(\bx)  & \bx \in \Omega, t > 0, \\			
			u(\bx, t) &=& 0 &\bx \in \partial \Omega, t > 0,\\			
			u(\bx, 0) &=& 0 &\bx \in \Omega
		\end{array}
	\right.
	\label{main eqn}
\end{equation}
where $\mathcal{A}$ is an elliptic operator independent of the time and $f(\bx, t) p(\bx)$ is the source function. 
The aim of this paper is to solve the following inverse source problem.
\begin{problem}[Inverse source problem for parabolic equations]
	Let $T$ be a positive number.
	Assume the function $f(\bx, t)$, $(\bx, t) \in \Omega \times [0, T]$, is known and $f(\bx, 0)$ does not vanish at any point $\bx$ in $\Omega$.
	Determine the function $p(\bx)$, $\bx \in \Omega$, from the measurement of the following data
	\begin{equation}
		G({\bx}, t) = \partial_{\nu}u(\bx, t) 
		\label{data}
	\end{equation}
	for all $\bx \in \partial\Omega$ and $t \in [0, T].$
	\label{ip}
\end{problem}

The uniqueness of Problem \ref{ip} when the source function is a combination of some Dirac functions  is confirmed in \cite{BadiaDuong:jiip2002} and a numerical method to reconstruct this source is studied in \cite{AndrleBadia:ipse2015}.
We also draw the reader to the conditional stability in \cite{Klibanov:ip2006, LiYamamotoZou:cpaa2009}.
In the case when the governing equation is the heat equation and the source function does not depend on the second variable, a reconstruction formula is provided in \cite{Malyshev:jmaa1989}.
Another related problem is the inverse problem of reconstructing the initial condition for parabolic equation. This problem is very important and interesting, see  \cite{LiNguyen:arxiv2019, Prilepko:pam2000, TuanKhoaAu:SIAM2019, klibanovYagola:arxiv2019, Tuan:ip2017} for theoretical results and numerical methods.
In the current paper, we introduce the following approach to solve Problem \ref{ip}. 
We derive from a governing equation a new equation involving only one unknown.
The solution to that equation will directly provide the knowledge of the desired source function.
However, that equation is not a standard partial differential equation. In fact, it involves the initial condition of itself.  
We prove the stability of the inverse source problem based on the projection of this equation on a finite dimensional space.
A theory to solve this partial differential equation is not available yet. 
To solve this equation, we employ the quasi-reversibility method.
This method was first introduced by Lattès and Lions \cite{LattesLions:e1969}. 
It is used to computed numerical solutions to ill-posed problems for partial differential equations. 
Due to its strength, since then, the quasi-reversibility method attracts the great attention of the scientific community
see e.g., \cite{Becacheelal:AIMS2015, Bourgeois:ip2006,
BourgeoisDarde:ip2010, BourgeoisPonomarevDarde:ipi2019, ClasonKlibanov:sjsc2007, Dadre:ipi2016, KaltenbacherRundell:ipi2019,
KlibanovSantosa:SIAMJAM1991, Klibanov:jiipp2013, LocNguyen:ip2019}.
We refer the reader to \cite{Klibanov:anm2015} for a survey
on this method. 
The solutions of partial differential equations due to the quasi-reversibility method are called
{\it regularized solution} in the theory of ill-posed problems \cite{Tihkonov:kapg1995}. 
The convergence of the regularized solution to the true one for three main types of partial differential equations is well-known  \cite{Klibanov:anm2015}. 
Recently, in \cite{LocNguyen:ip2019}, the second author proved a Lipchitz convergence of quasi-reversibility method for the hyperbolic operator that involves Volterra integrals.
The proof for a Lipchitz convergence of the quasi-reversibility method for the parabolic operator including the initial condition when this initial condition takes some particular forms will be proved in our near future publication.

An application of the inverse source problem in this paper is to solve a coefficient inverse problem for the heat equation.
Given an initial guess of the coefficient, we show that our inverse source problem is a linear ``perturbation" of that nonlinear coefficient inverse problem near that initial guess. Hence, by repeatedly solving our inverse source problem, we can obtain the solution to the coefficient inverse problem,
see Section \ref{sec cip} for details.
It is worth mentioning that the optimal control method to solve nonlinear coefficient inverse problem is widely used \cite{Borceaetal:ip2014, CaoLesnic:nmpde2018, CaoLesnic:amm2019, KeungZou:ip1998, YangYuDeng:amm2008} which provide good numerical results with reasonable initial guesses. 
We also refer the reader to \cite{KlibanovNik:ra2017, Klibanov:ip2015} for the convexification method and numerical results in 1D.

The paper is organized as follows. 
We propose an algorithm to solve Problem \ref{ip} in Section \ref{sec inv}. 
In section \ref{sec Lip}, we study the stability of Problem \ref{ip} in an approximation context.
Next, in Section \ref{sec implement}, we present the details about the implementation of our algorithm. 
In Section \ref{sec illu}, we show some numerical solutions to the inverse source problem.
In Section \ref{sec cip}, we solve the nonlinear coefficient inverse problem from which the inverse source problem above arises.
Section \ref{sec concluding} is for concluding remarks.

\section{The inversion method}
\label{sec inv}

Define the function 
\begin{equation}
	v(\bx, t) = u_t(\bx, t) \quad \mbox{for all } \bx \in \Omega, t \in (0, T).
	\label{def v}
\end{equation}
Since $\mathcal{A}$ does not depend on $t$,
it follows from the partial differential equation in \eqref{main eqn} that
\begin{equation}
	v_t(\bx, t) = \mathcal{A}v (\bx, t) + f_t(\bx, t) p(\bx) 
	\label{3.111111}
\end{equation}
for all $\bx \in \Omega,$ $t \in (0, T).$
The initial condition for the function $v$ can be computed as
\[
	v(\bx, 0) = u_t(\bx, 0) = f(\bx, 0) p(\bx),
\]
which implies
\begin{equation}
	p(\bx) = \frac{v(\bx, 0)}{f(\bx, 0)} \quad \mbox{for all } \bx \in \Omega.
	\label{px}
\end{equation} 
Substituting this into \eqref{3.111111}, we obtain 
\begin{equation}
	v_t(\bx, t) = \mathcal{A} v(\bx, t) + \frac{f_t(\bx, t)}{f(\bx, 0)} v(\bx, 0)
	\label{3.2}
\end{equation}
for all $\bx \in \Omega$, $t \in [0, T].$
Note that equation \eqref{3.2} does not depend on the function $p(\bx)$.

Problem \ref{ip} becomes the problem of computing the function $v$ that satisfies \eqref{3.2} and the lateral Cauchy conditions
\begin{equation}
	v(\bx, t) = 0 \quad \mbox{and } \quad \partial_{\nu} v(\bx, t) = G_t(\bx, t)
	\label{3.4}
\end{equation}
for all $\bx \in \partial \Omega, t \in [0, T].$

\begin{remark}
	We consider the function $ G_t(\bx, t)$ as our ``indirect" data.
	In this paper, we test our method with noisy data
	$
		G_t(\bx, t) = G_t(\bx, t)(1 + \delta(-1 + 2{\rm rand}))
	$ 
	where $\delta$ is the noise level and ${\rm rand}$ is the uniformly distributed random number taking values in $[0, 1].$ In this paper, $\delta = 0\%, 5\%$ and $10\%.$
\end{remark}

\begin{remark}
	Our main idea when deriving \eqref{3.2} is that we want to eliminate one unknown so that we can arrive at the situation of one unknown and one equation. 
	This strategy was applied in our research group in many publications; see e.g., \cite{KlibanovNguyen:ip2019, NguyenLiKlibanov:IPI2019, LocNguyen:ip2019}.
	Among them, the most similar idea to derive \eqref{3.2} is in \cite{LocNguyen:ip2019} when the source term of a hyperbolic equation is eliminated.
	The main difference from \eqref{3.2} is that the corresponding equation in \cite{LocNguyen:ip2019} is an integro-differential equation, which is not applicable in the current paper.
\end{remark}

Assume that $v$ is known. Then, the desired function $p$ is computed via \eqref{px}. 
However, due to the presence of the term $v(\bx, 0)$, equation \eqref{3.2}, together with the lateral data in \eqref{3.4}, is not a standard partial differential equation.
A theortical method to solve it is not yet available.
We solve \eqref{3.2} and \eqref{3.4} by the quasi-reversibility method.
Define the operator
\begin{equation}
	L v(\bx, t) = v_t(\bx, t) -  \mathcal{A} v(\bx, t) - \frac{g_t(\bx, t)}{g(\bx, 0)} v(\bx, 0)
	\label{def L}
\end{equation} 
for all function $v \in C^2(\overline \Omega \times [0, T])$. 
Given $\epsilon > 0,$ we minimize the functional
\begin{equation}
	J_{\epsilon}(v) = \int_{0}^T \int_{\Omega}  |Lv(\bx, t)|^2 d\bx dt + \epsilon \|v\|_{H^{2,1}( \Omega \times [0, 1])}^2
	\label{Jepsilon}
\end{equation}
subject to the constraints in \eqref{3.4}.

The following proposition guarantees that $J_\epsilon$ has a unique minimizer in $H$.
\begin{Proposition}
	Assume that the set 
	\[
		H = \left\{v \in H^{2,1}(\Omega \times (0, T)) \mbox{ that satisfies } \eqref{3.4}\right\}
	\]
	is nonempty.
	Then, for each $\epsilon > 0,$ the function $J_{\epsilon} $ has a unique minimizer in $H$.
	\label{prop 1}
\end{Proposition}

The proof of this proposition follows the proof of Proposition 3.1 in \cite{NguyenLiKlibanov:IPI2019} for the time independent case. 
We present the proof for the time dependent case here for the connivence of the reader. 
\begin{proof}[Proof of Proposition \ref{prop 1}]
	Let $\mathcal E$ be a function in $H$. 
	Denote by $H_{0}$ the space $H - \mathcal E.$
	Introduce $w = v - \mathcal E$. 
	Then, minimizing $J_{\epsilon}(v)$ for $v$ in $H$ is equivalent to minimizing $J_{\epsilon}(w + \mathcal E)$ for $w$ in $H_0.$
	If $w \in H_0$ is a minimizer of $J_{\epsilon}(w + \mathcal E)$ in $H_0$, then, by the variational principle, 
	\[
		\langle L(w + \mathcal E), L \phi\rangle_{L^2(\Omega \times [0, T])} 
		+ 
		\epsilon\langle w ,  \phi\rangle_{H^{2,1}(\Omega \times [0, T])} = 0,
	\]
	which is equivalent to
	\begin{multline}
		\langle Lw, L \phi\rangle_{L^2(\Omega \times [0, T])} 
		+ 
		\epsilon\langle w,  \phi\rangle_{H^{2,1}(\Omega \times [0, T])} 
		\\
		= -\langle L\mathcal E, L \phi\rangle_{L^2(\Omega \times [0, T])}  
		- \epsilon\langle \mathcal E,  \phi\rangle_{H^{2,1}(\Omega \times [0, T])}.
		\label{2.7}
	\end{multline}
	The left hand side of \eqref{2.7} defines a new inner product $\{\cdot, \cdot\}$ in $H^{2, 1}(\Omega \times [0, T])$.
	We have $\{w, w\} \geq \epsilon\|w\|_{H^{2, 1}(\Omega \times [0, T])}^2$ and $\{w, w\} \leq C \|w\|_{H^{2, 1}(\Omega \times [0, T])}^2$ for some constant $C$ due to the trace theory and the assumption that $\mathcal A$ is a second order elliptic operator.
	Hence, $\{\cdot, \cdot\}$ is equivalent to the standard inner product of $H^{2, 1}(\Omega \times [0, T])$. 
	On the other hand, the right hand side of \eqref{2.7} is a bounded linear operator defined on $H^{2, 1}(\Omega \times [0, T])$.
	The existence and the uniqueness of a function $w$ satisfying \eqref{2.7} follows from the Riesz representation theorem.
	\end{proof}
\begin{remark}
	The unique minimizer of $J_{\epsilon}$ is call the regularized solution to \eqref{3.2} and \eqref{3.4}.
\end{remark}

Our method to solve Problem \ref{ip} is summarized in Algorithm \ref{alg}.
In practice, we implement Algorithm \ref{alg} in the finite difference scheme. 
We present the implementation of Algorithm \ref{alg} with the finite difference method in the Section \ref{sec implement}.
\begin{algorithm}[h!]
\caption{\label{alg}The procedure to solve Problem \ref{ip}}
	\begin{algorithmic}[1]
	\State\, Compute the Neuman data $G_t(\bx, t)$  for all $\bx \in \partial \Omega,$ $t \in [0, T].$
	\State\, Solve \eqref{3.2} and \eqref{3.4} by the quasi-reversibility method; i.e., minimizing $J_{\epsilon}$, $0 < \epsilon \ll 1$, subject to the constraints in \eqref{3.4}.
	The obtained minimizer is denoted by the function $v(\bx, t)$, $(\bx, t) \in \Omega \times [0, T]$.
	\State\, The desired source function $p(\bx)$ is computed by $\frac{v(\bx, 0)}{g(\bx, 0)}$, see \eqref{px}.
	\end{algorithmic}
\end{algorithm}

\section{A Lipschitz estimate based on a truncation of the Fourier series} \label{sec Lip}

Let $\{\Psi_n\}_{n = 1}^{\infty}$ be an orthonormal basis of $L^2(0, T).$ 
For each $\bx \in \Omega$, we can write 
\begin{equation}
	v(\bx, t) = \sum_{n = 1}^{\infty}v_n(\bx) \Psi_n(t) \quad \mbox{for all } (\bx, t) \in \Omega \times [0, T]
	\label{series}
\end{equation}
where $v(\bx, t)$ is the function defined in \eqref{def v}.
Here,
\begin{equation}
	v_n(\bx) = \int_0^T v(\bx, t) \Psi_n(\bx, t) dt \quad \mbox{for all } \bx \in \Omega.
	\label{def vn}
\end{equation}
Approximate the series in \eqref{series} by
\begin{equation}
	v(\bx, t) = \sum_{n = 1}^N v_n(\bx) \Psi_n(t)
	\label{v trunc}
\end{equation}
for  $(\bx, t) \in \Omega \times [0, T]$ for some number $N > 0$.
We also write
\begin{equation}
	v_t(\bx, t) = \sum_{n = 1}^N v_n(\bx) \Psi_n'(t)
	\label{vt trunc}
\end{equation}
Plugging \eqref{v trunc} and \eqref{vt trunc} into \eqref{3.2}, we have
\[
	\sum_{n = 1}^N v_n(\bx) \Psi_n'(t) = \sum_{n = 1}^N \mathcal A v_n(\bx) \Psi_n(t) + \frac{f_t(\bx, t)}{f(\bx, 0)} \sum_{n = 1}^N v_n(\bx) \Psi_n(0).
\]
Multiply both side of the equation above by $\Psi_m(t)$ for each $m \in \{1, \dots, N\}$ and then integrate the resulting equation on $[0, T]$. 
We obtain
\begin{multline}
	\sum_{n = 1}^N v_n(\bx) \int_{0}^T \Psi_m(t)\Psi_n'(t) dt = \sum_{n = 1}^N \mathcal A v_n(\bx)\int_0^T \Psi_m \Psi_n(t)dt
	\\
	 + \sum_{n = 1}^Nv_n(\bx) \Psi_n(0) \int_0^T \frac{f_t(\bx, t)}{f(\bx, 0)} \Psi_m(t) dt
	 \label{eqn for vn}
\end{multline}
for all $\bx \in \Omega.$
Define $V(\bx) = (v_1(\bx), \dots, v_N(\bx))$. 
It follows by \eqref{eqn for vn} and the fact that $\Psi_m$ that the vector valued function $V(\bx)$ satisfies the system 
\begin{equation}
	\mathcal A V = S V
\end{equation}
where $S$ is a $d \times d$ matrix valued function given by
\[
	S = \Big(\int_0^T (\Psi_m(t) \Psi'_n(t) - \frac{f_t(\bx, t)}{f(\bx, 0)}\Psi_n(0)\Psi_m(t)dt \Big)_{m,n = 1}^{\infty}.
\]
Since $f$ is a smooth function, so is $S$. By a standard compact argument for elliptic equation, we can find a constant $C$ depending only on $\mathcal A$, $N$ and $\Omega$ such that
\begin{equation}
	\|V\|_{H^1(\Omega)^N} \leq C [\|V\|_{H^{1/2}(\partial \Omega)^N} + \|\partial_\nu V\|_{H^{-1/2}(\partial \Omega)^N}].
	\label{est V}
\end{equation}
It follows from \eqref{main eqn}, \eqref{data}, \eqref{def v} and \eqref{def vn} that
\begin{align*}
	V(\bx) &= 0, \\
	\\ 
	\partial_{\nu} V(\bx) &= \Big(\int_0^T G_t(\bx, t) \Psi_n(\bx, t) dt\Big)_{n = 1}^N \\
	&= \Big(G(\bx, T) - G(\bx, 0) - \int_0^T G(\bx, t) \Psi_n(\bx, t) dt\Big)_{n = 1}^N
\end{align*}
on $\partial \Omega$.
Hence, by \eqref{est V},
\begin{equation}
	\|V\|_{H^1(\Omega)^N} 
	\leq C [\||G(\cdot, T)| + |G(\cdot, 0)|\|_{H^{-1/2}(\partial \Omega)} + \|G\|_{H^{-1/2, 1}(\partial \Omega \times [0, T])}].
	\label{est V1}
\end{equation}
Using \eqref{v trunc}, we have
\begin{equation}
	\|v\|_{H^{2, 1}(\Omega)} C [\||G(\cdot, T)| + |G(\cdot, 0)|\|_{H^{-1/2}(\partial \Omega)} + \|G\|_{H^{-1/2, 1}(\partial \Omega \times [0, T])}].
\end{equation}
As a result, using \eqref{px} and the trace theory, we get
\[
	\|p\|_{L^2(\Omega)} \leq C [\||G(\cdot, T)| + |G(\cdot, 0)|\|_{H^{-1/2}(\partial \Omega)} + \|G\|_{H^{-1/2, 1}(\partial \Omega \times [0, T])}].
\]

In summary, we have proved the following theorem.
\begin{Theorem}
	Assume that the function $v(\bx, t) = u_t(\bx, t)$ is well-approximated by the Fourier sum as in \eqref{v trunc} for some integer $N$ where $u(\bx, t)$ is the solution to \eqref{main eqn}. 
	Then, there exists a constant $C$ depending only on $\mathcal A$, $N$ and $\Omega$ such that
	\[
		\|p\|_{L^2(\Omega)} \leq C [\||G(\cdot, T)| + |G(\cdot, 0)|\|_{H^{-1/2}(\partial \Omega)} + \|G\|_{H^{-1/2, 1}(\partial \Omega \times [0, T])}].
	\]
	\label{main thm}
\end{Theorem}

Theorem \ref{main thm} implies the Lipschitz stability for Problem \ref{ip} in the finite dimensional space spanned by $\{\Psi_1, \dots, \Psi_N\}$.
	Studying the stability when $N$ tends to $\infty$ is extremely challenging and is out of the scope of this paper.

\begin{remark}
	The assumption about the well-approximation in Theorem \ref{main thm} is verified numerically in some recent works of our research group. 
	This verification for elliptic equation can be found in \cite{NguyenLiKlibanov:IPI2019} and the one for parabolic equation is in \cite{LiNguyen:arxiv2019}.
	In those papers, the basis $\{\Psi_n\}_{n = 1}^{\infty}$ is taken from \cite{Klibanov:jiip2017}.
\end{remark}

\section{The finite difference method to find the regularized solution} \label{sec implement}

In this section, the domain $\Omega$ is set to be a square in $
\R^2$; i.e,
\[
	\Omega = (-R, R)^2 
\]
where $R$ is a positive number.
Let $N_\bx$ and $N_t$ be positive integers. Set $d_\bx = 2R/N_\bx$ and $d_t = T/Nt.$
We define a set of grid points on $\overline \Omega$
\[
	\mathcal{G} = \left\{
		(x_i, y_j) = (-R + (i - 1)d_\bx, -R + (j - 1)d_\bx): 1 \leq i, j \leq N+1
	\right\}
\]
and define a uniform partition on the time domain $[0, T]$ as
\[
	0 = t_1  < t_2 < \dots < t_{N_t + 1}, \quad t_l = (l - 1)d_t, 1 \leq l \leq N_t + 1.
\] 
For the simplicity in implementation, in this section, we modify the $H^{2,1}$ norm in the regularization term in \eqref{Jepsilon} to the $H^1$ norm.
In other words,
\begin{equation*}
	J_{\epsilon}(v) = \int_{\Omega} \int_0^T |Lv|^2 d\bx + \epsilon \int_0^T \int_\Omega \left(|v|^2 + |\nabla v|^2\right)d\bx 
\end{equation*}
for all $v \in H^{2,1}(\Omega \times [0, T]).$
\begin{remark}
We replace the norm in regularization term $\epsilon \|v\|_{H^{2, 1}(\Omega \times T)}^2$ by the $H^1$--norm 
because the $H^1$--norm is easier to implement. 
On the other hand, we have not observed any instabilities probably because the number $100 \times 100$ of grid points we use is not too large and all norms in finite dimensional spaces are equivalent. 
\end{remark}
\begin{remark}[The choice of $\epsilon$]
	We observe numerically that if $\epsilon$ is larger than $10^{-5}$, the reconstructed images of the source function are good but the reconstructed values are low and if $\epsilon < 10^{-9}$, our method breaks down. We choose $\epsilon = 10^{-8}$ in all our numerical tests. 
	Note that this choice of $\epsilon$ is independent of the noise level, which is, in practice, supposed to be unknown.
\end{remark}
The finite difference version of $J_{\epsilon},$ still named as $J_{\epsilon}$, reads
\begin{multline}
	J_{\epsilon}(v) = d_t d_\bx^2\sum_{l = 2}^{N_t + 1} \sum_{i,j = 2}^{N_\bx} |L^{d_\bx, d_t} v(x_i, y_j, t_l)|^2 
	\\
	+ \epsilon d_t d_\bx^2\sum_{l = 2}^{N_t + 1} \sum_{i,j = 2}^{N_\bx }\left( |v(x_i, y_j, t_l)|^2 + |\nabla^{d_\bx} v(x_i, y_j, t_l)|^2\right).
	\label{Jepsilon FD}
\end{multline}
Here, $L^{d_\bx, d_t}$ is the  approximation of $L$ in the finite difference scheme and $\nabla^{d_\bx}$ is the finite difference gradient. 
From now on, for the simplicity and to minimize the effort of writing computational code, we consider the case 
\[
	\mathcal{A} v(\bx, t) = \Delta v(\bx, t) + c(\bx) v(\bx, t) 
\]
for some function $c$ in $L^{\infty}(\Omega).$
In this case,
\begin{multline}
	L^{d_\bx, d_t} v(x_i, y_j, t_l) = \frac{v(x_i, y_j, t_l) - v(x_i, y_j, t_{l - 1})}{d_t} 
	\\
	 - \frac{v(x_{i+1}, y_j, t_l) + v(x_{i - 1}, y_j, t_l) +v(x_i, y_{j + 1}, t_l) + v(x_i, y_{j-1}, t_l) - 4v(x_i, y_j, t_l)}{d_\bx^2}
	 \\
	 - c(x_i, y_j)v(x_i, y_j, t_l) - \frac{f_t(x_i, y_j, t_l)}{f(x_i, y_j, t_1)} v(x_i, y_j, t_1)
	 \label{3.1}
\end{multline}
and
\[
	\nabla^{d_\bx} u(x_i, y_j, t_l) = \left(\frac{u(x_{i + 1}, y_j, t_k) - u(x_{i}, y_j, t_k)}{d_{\bx}}, \frac{u(x_{i}, y_{j + 1}, t_k) - u(x_{i}, y_j, t_k)}{d_{\bx}}\right)
\]
for all $1 \leq i, j \leq N_{\bx}$ and $1 \leq k \leq N_t + 1.$
Introduce the $N$ dimensional vector $\mathfrak{v}$, $N = (N_{\bx} + 1)^2 (N_t + 1)$, whose $n^{\rm th}$ entry is given by
\begin{equation}
	\mathfrak{v}_n = v(x_i, y_j, t_l) 
	\label{lineup}
\end{equation}
where 
\[
	n = (i-1)(N_{\bx} + 1)(N_t + 1) + (j - 1)(N_t + 1) + l, \quad 1 \leq i, j \leq N_{\bx} + 1, 1 \leq l \leq N_t+1.
\]
Then, we can rewrite \eqref{3.1} as
\begin{equation}
	L^{d_\bx, d_t} v = D \mathfrak{v}
	\label{matrix D}
\end{equation}
where the $N \times N$ matrix $D$ is described as follows. 
For each $n = (i-1)(N_{\bx} + 1)(N_t + 1) + (j - 1)(N_t + 1) + l$  with $2 \leq i, j \leq N_{\bx}$ and $2 \leq l \leq N_t$,
\begin{enumerate}
	\item the $nn^{\rm th}$ entry $D_{nn}$ is given by $\frac{1}{d_t} + \frac{4}{d_{\bx}^2} - c(x_i, y_j)$;
	\item the $n m^{\rm th}$ entry $D_{n m}$ is given by $-\frac{1}{d_t}$ where $m = (i-1)(N_{\bx} + 1)(N_t + 1) + (j - 1)(N_t + 1) + l - 1$ for $3 \leq l \leq N_t$;
	\item the $n m^{\rm th}$ entry $D_{n m}$ is given by $-\frac{1}{d_t} - \frac{f_t(x_i, y_j, t_l)}{f(x_i, y_j, t_1)}$ where $m = (i-1)(N_{\bx} + 1)(N_t + 1) + (j - 1)(N_t + 1) + l - 1$ for $ l = 2$;
	\item the $n m^{\rm th}$ entry $D_{n m}$ is given by $-\frac{1}{d_{\bx}^2}$ where $m = (i \pm 1 -1)(N_{\bx} + 1)(N_t + 1) + (j \pm 1 - 1)(N_t + 1) + l - 1$ for $2 \leq l \leq N_t$;
	\item the other entries are $0$.
\end{enumerate}
We next define the matrices $D_x$ and $D_y$ such that $(D_x \mathfrak{v}, D_y \mathfrak{v}) = \nabla^{d_\bx} v$. 
For each $n = (i-1)(N_{\bx} + 1)(N_t + 1) + (j - 1)(N_t + 1) + l$  with $2 \leq i, j \leq N_{\bx} + 1$ and $1 \leq l \leq N_t + 1$,
\begin{enumerate}
	\item the $nn^{\rm th}$ entry of $D_x$ and $D_y$ is given by $\frac{1}{d_\bx}$;
	\item the $nm^{\rm th}$ entry of $D_x$ is given by $-\frac{1}{d_\bx}$ for $m = (i - 1 -1)(N_{\bx} + 1)(N_t + 1) + (j - 1)(N_t + 1) + l$;
	\item the $nm^{\rm th}$ entry of $D_y$ is given by $-\frac{1}{d_\bx}$ for $m = (i  -1)(N_{\bx} + 1)(N_t + 1) + (j - 1 - 1)(N_t + 1) + l$;
	\item other entries are $0.$
\end{enumerate}
The finite difference version of $J_{\epsilon}$, defined in \eqref{Jepsilon FD}, becomes
\[
	J_{\epsilon} v = d_td_{\bx}^2\left[|D \mathfrak{v}|^2 + \epsilon \left( |v^2| + |D_x \mathfrak {v}|^2 + |D_y \mathfrak {v}|^2 \right) \right].
\]
Hence, due to \eqref{lineup}, since $v$ is a minimizer of $J_{\epsilon}$, $\mathfrak{v}$ satisfies the equation
\begin{equation}
	\left[D^T D + \epsilon \left(\Id + D_x^TD_x + D_y^TD_y\right)\right] \mathfrak{v} = \vec{\bf 0}.
	\label{minimization}
\end{equation}

We next consider the boundary conditions for $v$ in \eqref{3.4}. 
In the finite difference scheme, the first condition in \eqref{3.4} reads for $l = 1, 2, \dots, N_t+1,$
\[
	v(x_i, y_j, t_l) = 0 
\]
for all $i \in \{1, N_{\bx}+1\}$ and  $j \in \{1, 2, \dots, N_{\bx} + 1\}$ or $i \in \{1, \dots, N_{\bx}+1\}$ and  $j \in \{1, N_{\bx} + 1\}$.
Therefore, due to \eqref{lineup}, we can write this condition as 
\begin{equation}
	K_1 \mathfrak{v} = \vec{\bf 0}
	\label{matrix K1}
\end{equation}
where $K_1$ is defined as follows. 
For $l \in \{1, 2, \dots, N_t + 1\},$
\begin{enumerate}
	\item the $nn^{\rm th}$ entry of $K_1$ is set to be $1$ if $n = (i-1)(N_{\bx} + 1)(N_t + 1) + (j - 1)(N_t + 1) + l$ for some $i \in \{1, N_\bx +1\}$, $j \in \{1, 2, \dots, N_{\bx}+1\}$ or $i \in \{1, 2, \dots, N_\bx +1\}$, $j \in \{1, N_{\bx}+1\}$;
	\item the other entries of $K_1$ are $0$.
\end{enumerate}
The second condition in \eqref{3.4} is rewritten as
\begin{equation}
	K_2 \mathfrak{v} = \mathfrak{g}
	\label{matrix K2}
\end{equation}
where the vector $\mathfrak{g}$ is the lineup version of the data $G_t$
\[
	\mathfrak{g}_n = G_t(x_i, y_j, t_l) \quad n = (i-1)(N_{\bx} + 1)(N_t + 1) + (j - 1)(N_t + 1) + l
\]
for all $i \in \{1, N_{\bx}+1\}$ and  $j \in \{1, 2, \dots, N_{\bx} + 1\}$ or $i \in \{1, \dots, N_{\bx}+1\}$ and  $j \in \{1, N_{\bx} + 1\}$
and the matrix $K_2$ is defined as follows. For all $l \in \{1, 2, \dots, N_t+1\}$, 
\begin{enumerate}
	\item the $nn^{\rm th}$ entry of $K_2$ is $\frac{1}{d_\bx}$ if $n = (i-1)(N_{\bx} + 1)(N_t + 1) + (j - 1)(N_t + 1) + l$ for $i \in \{1, N_\bx +1\}$, $j \in \{1, 2, \dots, N_{\bx}+1\}$ or $i \in \{1, 2, \dots, N_\bx +1\}$, $j \in \{1,  N_{\bx}+1\}$;
	\item the $nm^{\rm th}$ entry of $K_2$ is $-\frac{1}{d_\bx}$ if $n = (i-1)(N_{\bx} + 1)(N_t + 1) + (j - 1)(N_t + 1) + l$ for $i = 1$, $j \in \{1, 2, \dots, N_{\bx}+1\}$ and $m = (i + 1-1)(N_{\bx} + 1)(N_t + 1) + (j - 1)(N_t + 1) + l$;
	\item the $nm^{\rm th}$ entry of $K_2$ is $-\frac{1}{d_\bx}$ if $n = (i-1)(N_{\bx} + 1)(N_t + 1) + (j - 1)(N_t + 1) + l$ for $i = N_\bx + 1$, $j \in \{1, 2, \dots, N_{\bx}+1\}$ and $m = (i - 1-1)(N_{\bx} + 1)(N_t + 1) + (j - 1)(N_t + 1) + l$;
	\item the $nm^{\rm th}$ entry of $K_2$ is $-\frac{1}{d_\bx}$ if $n = (i-1)(N_{\bx} + 1)(N_t + 1) + (j - 1)(N_t + 1) + l$ for $i \in \{2, \dots, N_{\bx}\}$, $j  = 1$ and $m = (i -1)(N_{\bx} + 1)(N_t + 1) + (j + 1 - 1)(N_t + 1) + l$;
	\item the $nm^{\rm th}$ entry of $K_2$ is $-\frac{1}{d_\bx}$ if $n = (i-1)(N_{\bx} + 1)(N_t + 1) + (j - 1)(N_t + 1) + l$ for $i \in \{ 2, \dots, N_{\bx}\}$, $j  = N_\bx+1$ and $m = (i -1)(N_{\bx} + 1)(N_t + 1) + (j - 1 - 1)(N_t + 1) + l$;
	\item the other entries of $K_2$ are 0.
\end{enumerate}
Combining \eqref{minimization}, \eqref{matrix K1} and \eqref{matrix K2}, we obtain
\[
	\left[
		\begin{array}{c}
			D^T D + \epsilon \left(\Id + D_x^TD_x + D_y^TD_y\right)
			\\
			K_1\\
			K_2 
		\end{array}
	\right]\mathfrak{v} = 
	\left[
		\begin{array}{c}
			\vec{\bf 0}\\
			\vec{\bf 0}\\
			\mathfrak{g}
		\end{array}
	\right]
\]
Since $\epsilon$ is a small number, it is acceptable that we modify the equation above by a more ``stable" one
\begin{multline}
	\left(
	\left[
		\begin{array}{c}
			 D 
			\\
			K_1\\
			K_2 
		\end{array}
	\right]^T 
	\left[
		\begin{array}{c}
			 D 
			\\
			K_1\\
			K_2 
		\end{array}
	\right] + \epsilon \left(\Id + D_x^TD_x + D_y^TD_y\right) \right)\mathfrak{v} 
	\\
	= 
	\left[
		\begin{array}{c}
			 D 
			\\
			K_1\\
			K_2 
		\end{array}
	\right]^T
	\left[
		\begin{array}{c}
			\vec{\bf 0}\\
			\vec{\bf 0}\\
			\mathfrak{g}
		\end{array}
	\right].
	\label{eqn for lineup v}
\end{multline}

The analysis in this section is summarized in the following proposition.
\begin{Proposition} The source function $p(\bx)$ can be computed by 
\begin{enumerate}
	\item solve \eqref{eqn for lineup v} for $\mathfrak{v}$,  the ``line up" version of $v$;
	\item compute the function $v_{\rm comp}$ using \eqref{lineup};
	\item calculate $p_{\rm comp}(\bx) = \frac{v_{\rm comp}(\bx)}{f(\bx, 0)}$.
\end{enumerate}
\end{Proposition}

\section{Numerical results}\label{sec illu}

We test our numerical method when $R = 1$ and $\Omega,$ therefore, is $(-1, 1)^2.$
Also, we choose $T = 0.2$, see the Remark \ref{rem T} for this choice of $T$.
\begin{remark}[Choose $T$]
	We numerically choose $T$ by examining the $L^2$ norm of the data $G_t(\bx, t)$ on $\partial \Omega$ as a function in $T$. 
	Define \[\gamma (t) = \|G_t(\cdot, t)\|_{L^2(\partial \Omega)}.\] 
	The graph of the function $\gamma$ is displayed in Figure \ref{fig T}, showing that the data  is largest on $[0, 0.2]$. This means the data contains most important information about the source in this interval. 
	We therefore choose $T = 0.2$ for all of our numerical tests.
\begin{figure}[h!]
		\begin{center}
			\includegraphics[width = 0.4\textwidth]{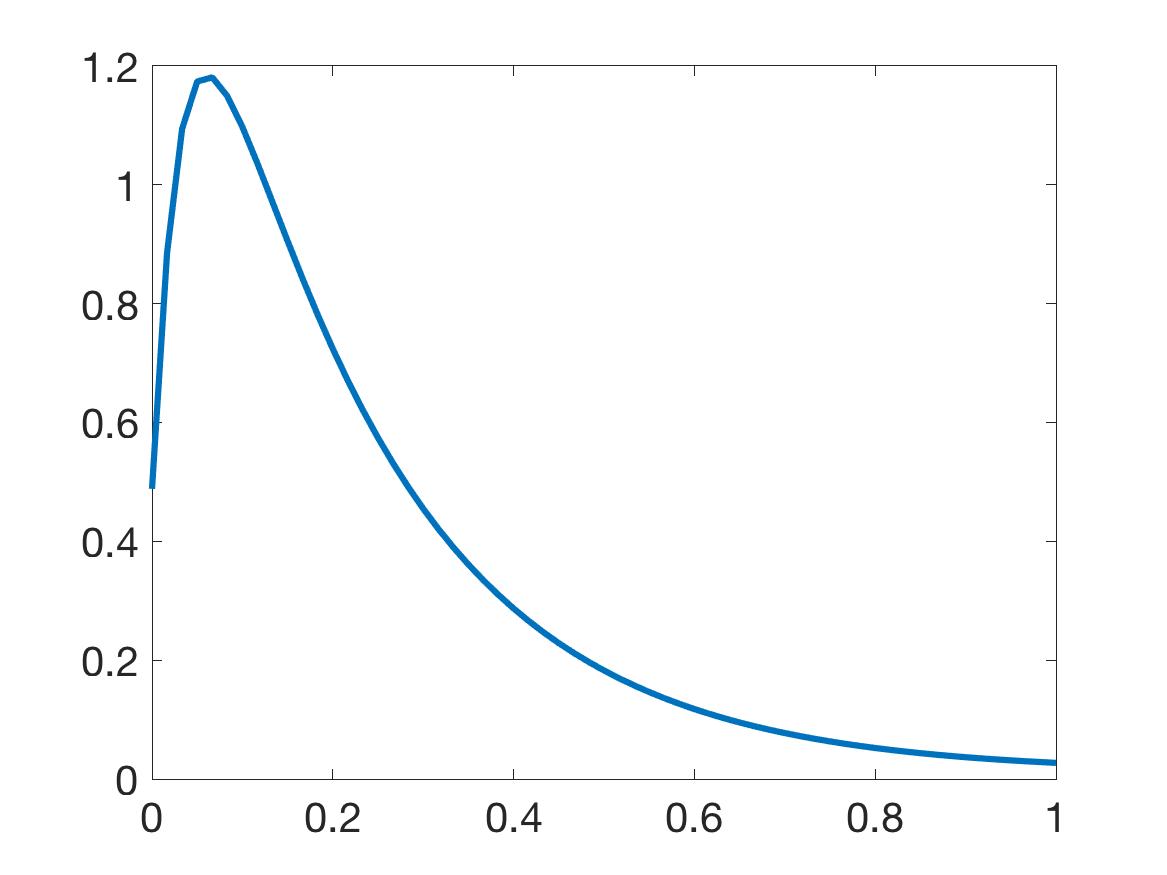}
			\caption{\label{fig T} The graph of the function $t \mapsto \|G_t(\cdot, t)\|_{L^2(\partial \Omega)}$ where $G$ is the function computed from the true source function in Test 1. We observe that the indirect data $G_t$ contains most information on $(0, 0.2)$.}
		\end{center}
	\end{figure}
	\label{rem T}
\end{remark}
We chose $N_{\bx} = 100$ and $N_t = 60$ in this section.
In all tests, the known function $f$ is chosen as
\[
	f(\bx, t) = 1 + 0.2e^{t|\bx|^2} \quad \bx \in \Omega, t \in [0, T]
\]
and the known function $c(\bx)$ is set to be
\[
	c(\bx) = 0.2 |\bx|^2 \quad \bx \in \Omega.
\]

In this section, we show the following numerical results.

\noindent {\bf Test 1}.
In this test, the true source function $p_{\rm true}$ is smooth and given by
\[
	p_{\rm true} = \left\{
		\begin{array}{ll}
			\exp\left(\frac{r^2}{r^2 - 0.5^2}\right) & \mbox{if }r = \sqrt{(x - 0.3)^2 + y^2} < 0.5
			\\
			0 &\mbox{otherwise.}
		\end{array}
	\right.
\]

The numerical result for this test is displayed in Figure \ref{fig Test 1}.
\begin{figure}[h!]
	\begin{center}
		\subfloat[The true source function]{
			\includegraphics[width = 0.45\textwidth]{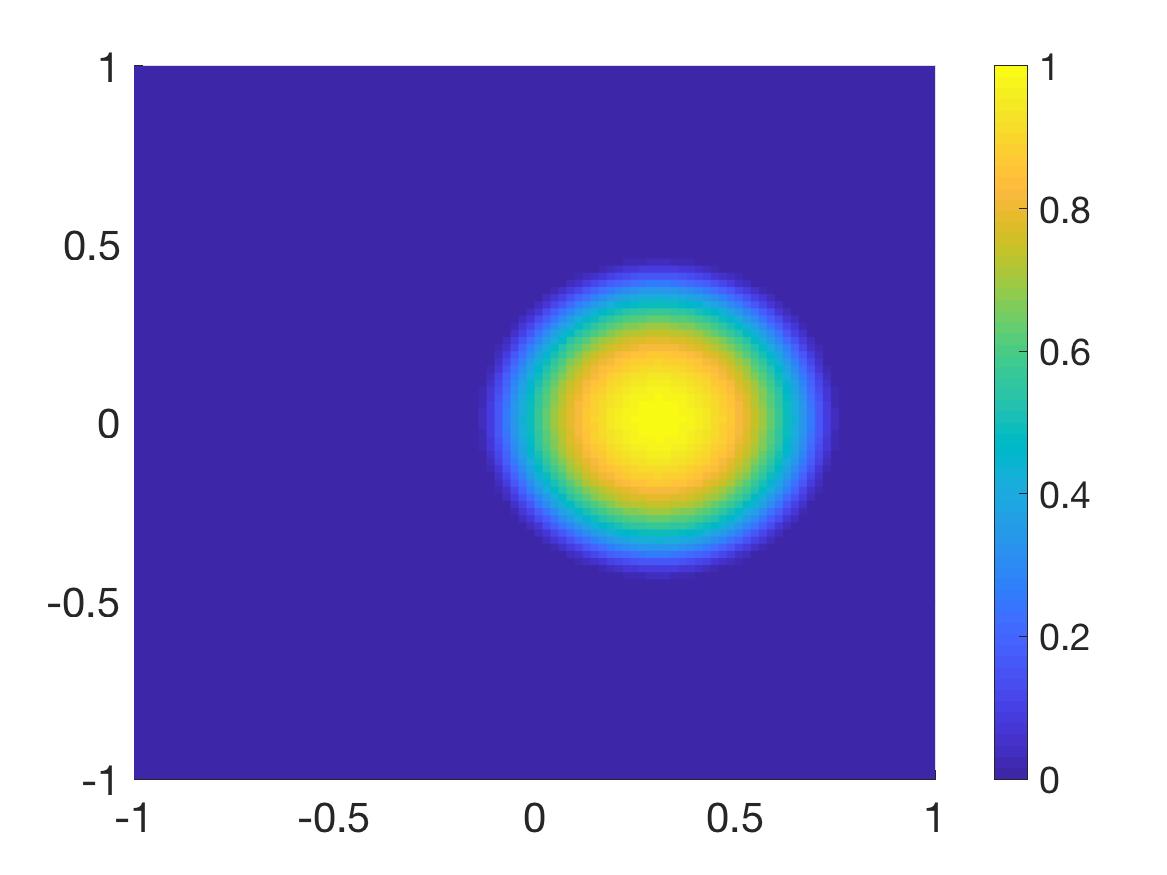}
		} \quad
		\subfloat[The computed source function, $\delta = 0\%$]{
			\includegraphics[width = 0.45\textwidth]{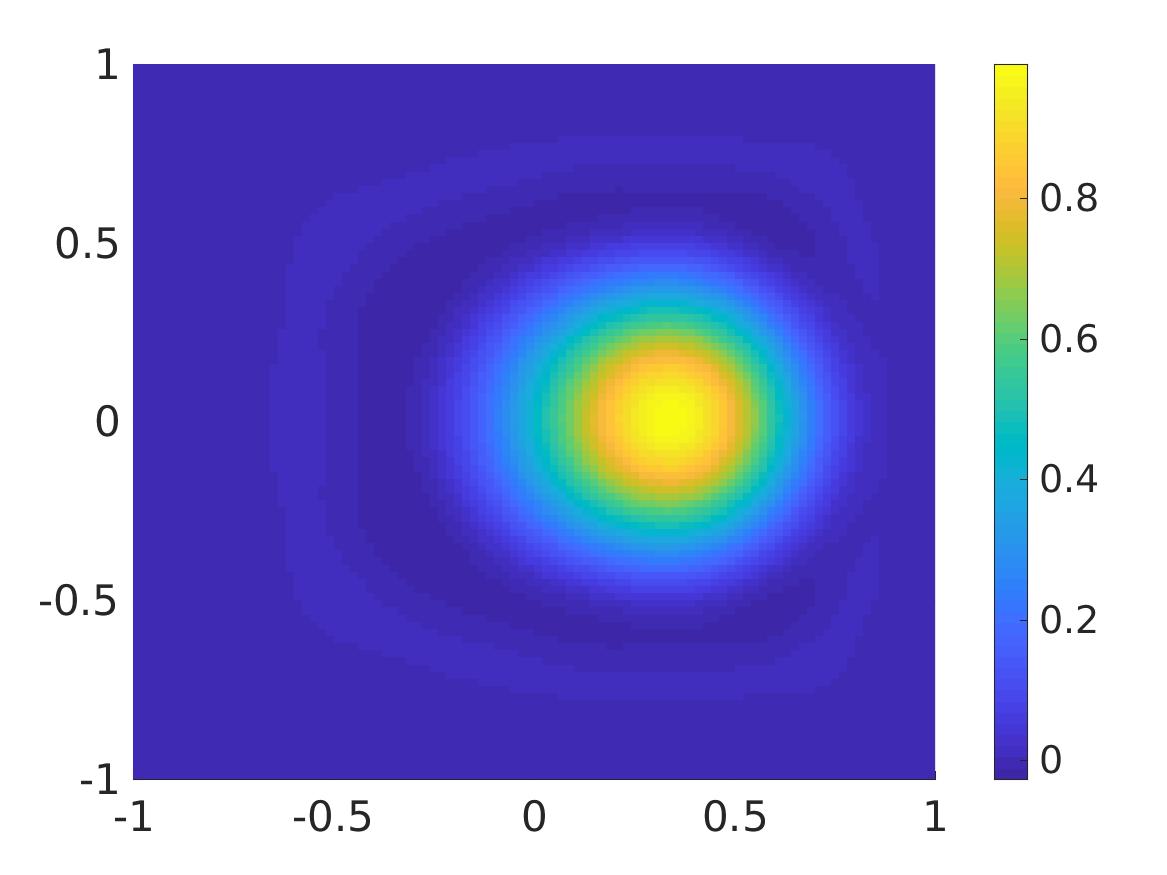}
		}
		
		\subfloat[The computed source function, $\delta = 5\%$]{
			\includegraphics[width = 0.45\textwidth]{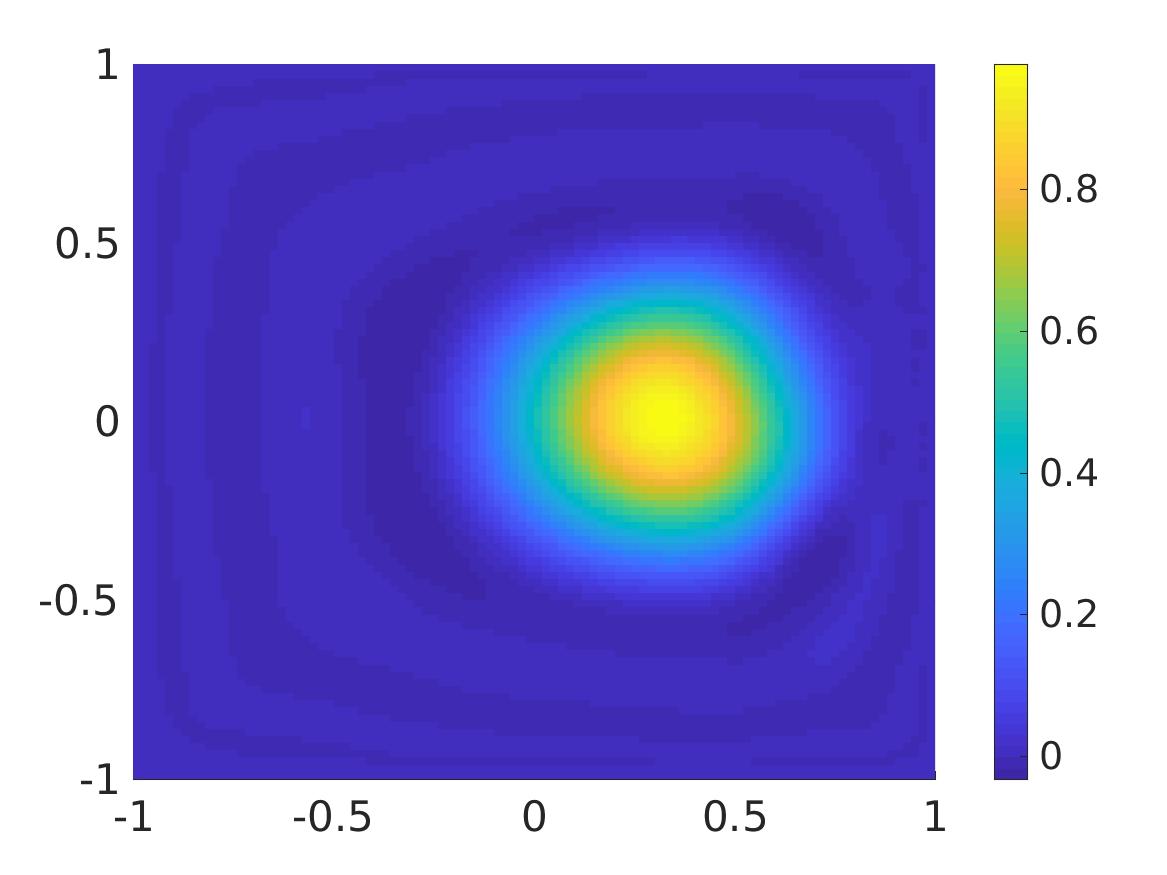}
		} \quad
		\subfloat[The computed source function, $\delta = 10\%$]{
			\includegraphics[width = 0.45\textwidth]{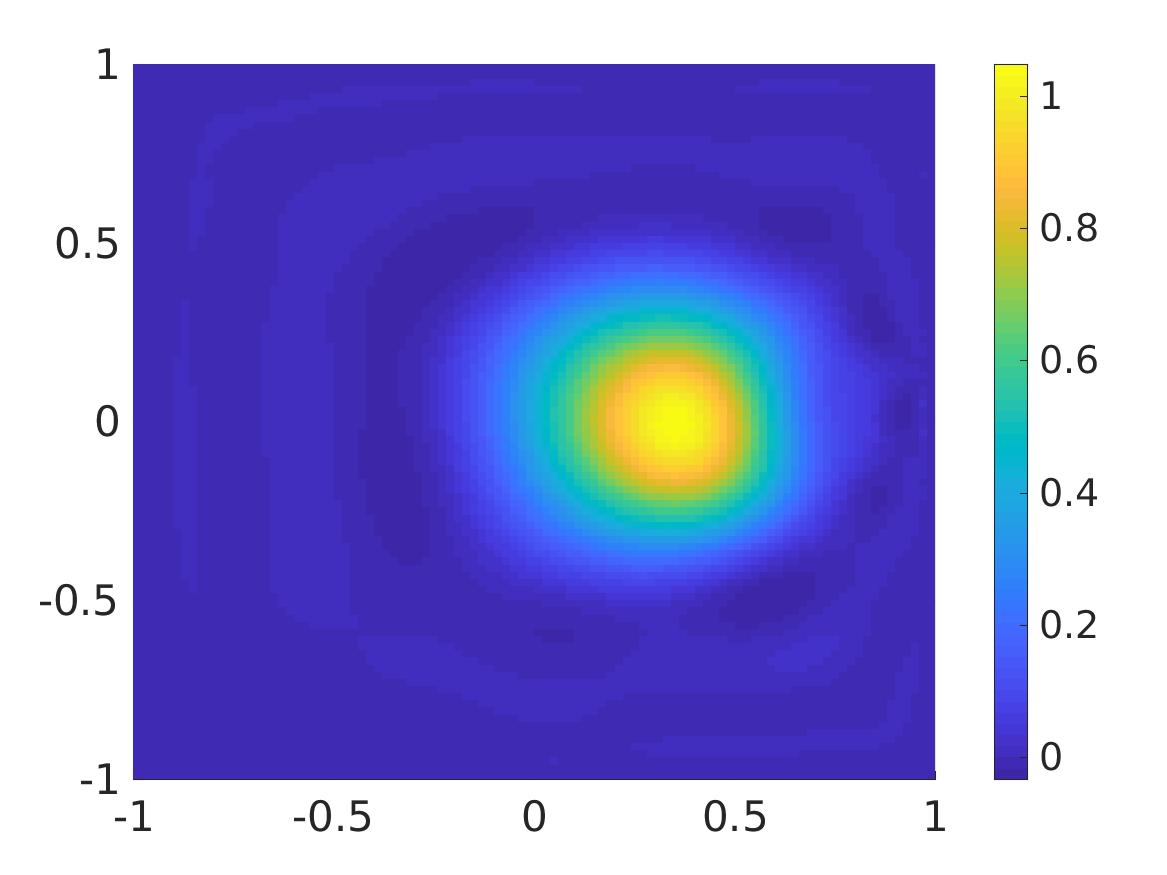}
		}
	\caption{\label{fig Test 1} Test 1. The true and computed source functions with different noise levels. The constructed source functions are quite accurate.}
	\end{center}
\end{figure}

It is evident that our method well reconstructs the source function $p_{\rm true}$.  
The location and shape of the circular ``inclusion" can be identified. 
The true maximum value of the inclusion is $1$.
The reconstructed maximum value of the inclusion is computed with small errors. 
In fact,
\begin{enumerate}
\item when $\delta = 0\%,$ $\max_{\bx \in \Omega} p_{\rm comp}(\bx) = 0.991$ and the coresponding relative error is $0.9\%$;
\item when $\delta = 5\%,$ $\max_{\bx \in \Omega} p_{\rm comp}(\bx) = 0.976$ and the coresponding relative error is $2.4\%$; 
\item when $\delta = 10\%,$ $\max_{\bx \in \Omega} p_{\rm comp}(\bx) = 1.048$ and the coresponding relative error is $4.8\%$. 
\end{enumerate}

\noindent {\bf Test 2.} We test our method for the case when $p_{\rm true}$ is given by the smooth function
\[
	p_{\rm true} = \left\{
		\begin{array}{ll}
			\exp\left(\frac{r_1^2}{r_1^2 - 0.5^2}\right) & \mbox{if }r_1 = \sqrt{(x - 0.4)^2 + (y-0.4)^2} < 0.5
			\\
			-\exp\left(\frac{r_2^2}{r_2^2 - 0.5^2}\right) & \mbox{if }r_2 = \sqrt{(x + 0.4)^2 + (y + 0.4)^2} < 0.5
			\\
			0 &\mbox{otherwise.}
		\end{array}
	\right.
\]
In this test, the true source function has a negative ``inclusion" and a positive one. 
The numerical results for this test are displayed in Figure \ref{fig Test 2}.
\begin{figure}[h!]
	\begin{center}
		\subfloat[The true source function]{
			\includegraphics[width = 0.45\textwidth]{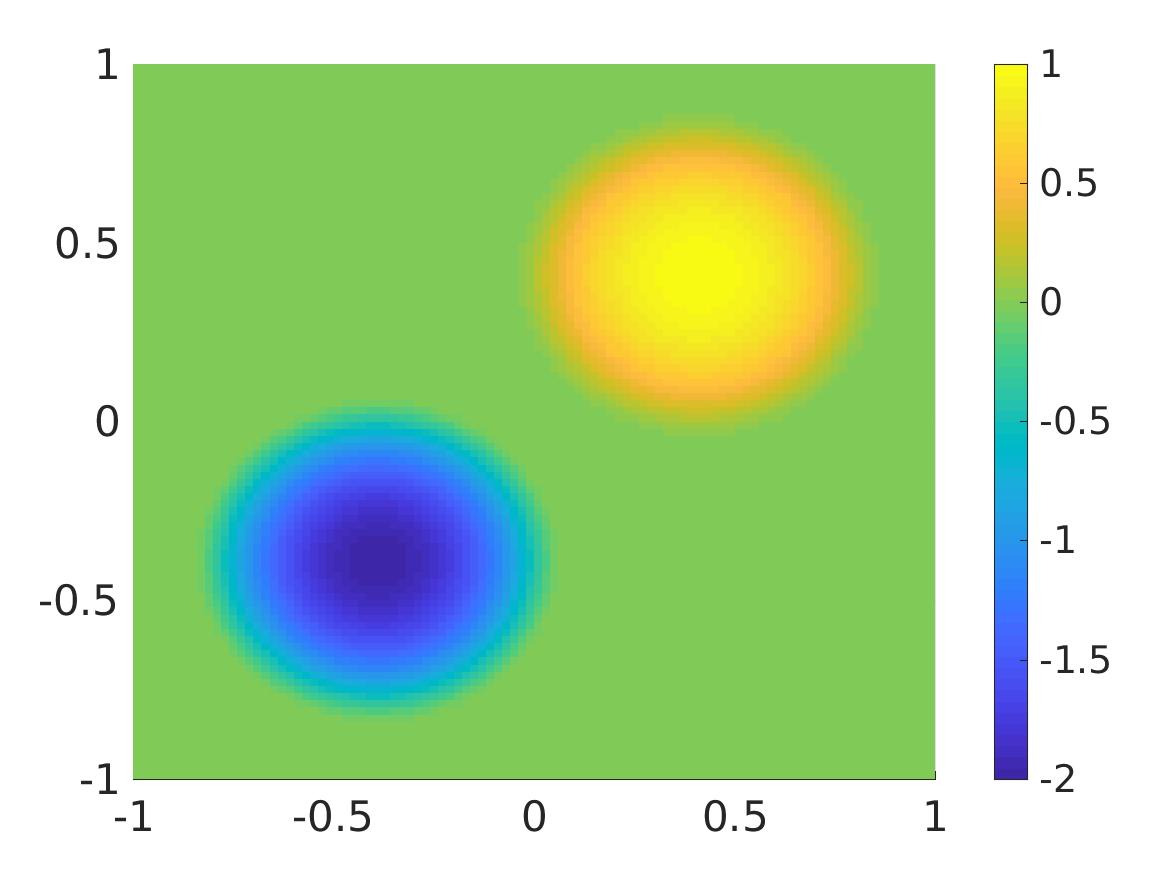}
		} \quad
		\subfloat[The computed source function, $\delta = 0\%$]{
			\includegraphics[width = 0.45\textwidth]{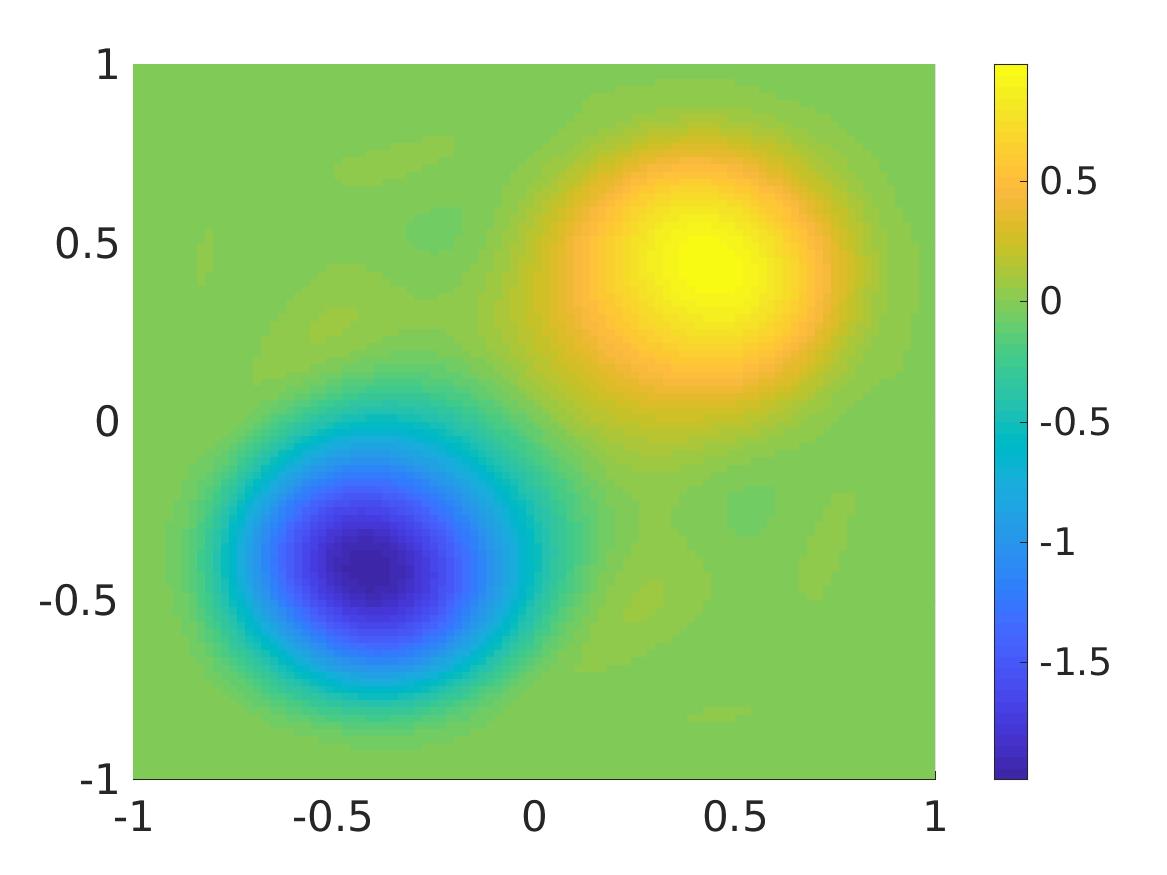}
		}
		
		\subfloat[The computed source function, $\delta = 5\%$]{
			\includegraphics[width = 0.45\textwidth]{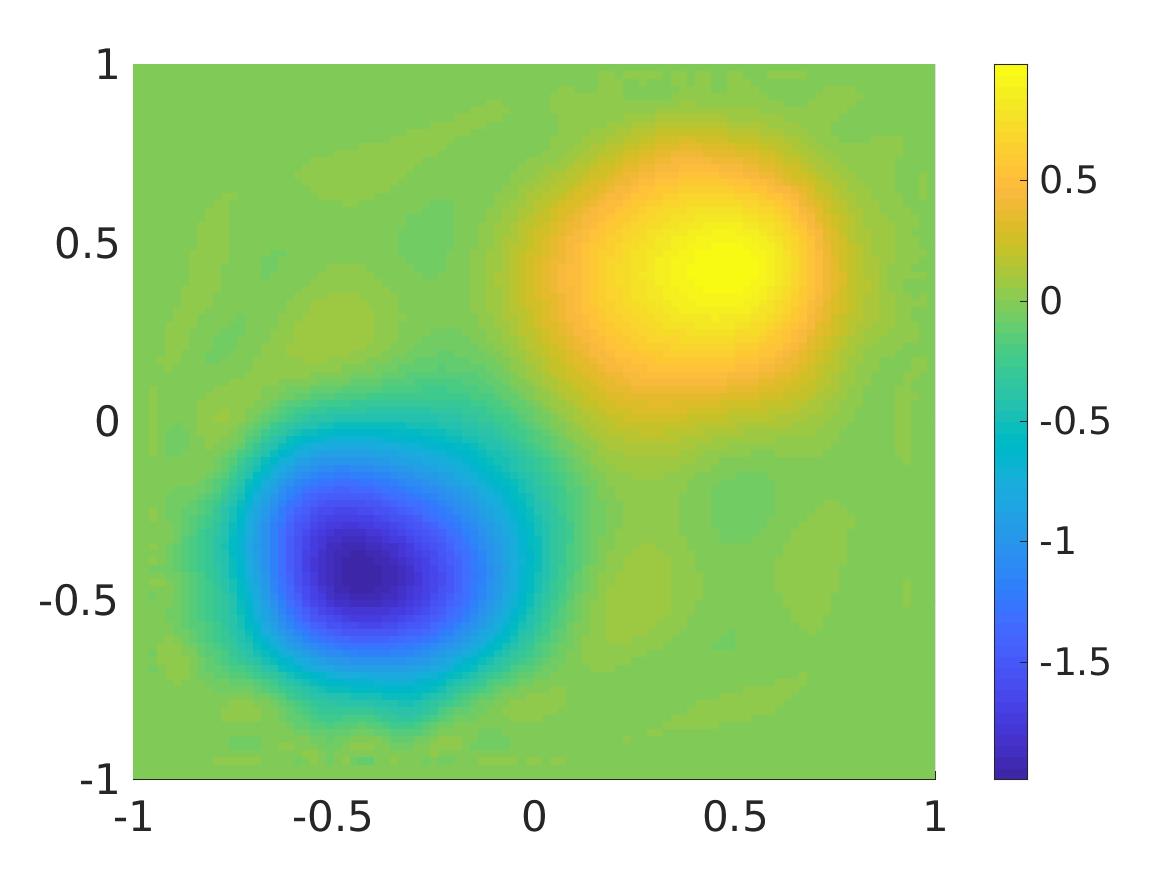}
		} \quad
		\subfloat[The computed source function, $\delta = 10\%$]{
			\includegraphics[width = 0.45\textwidth]{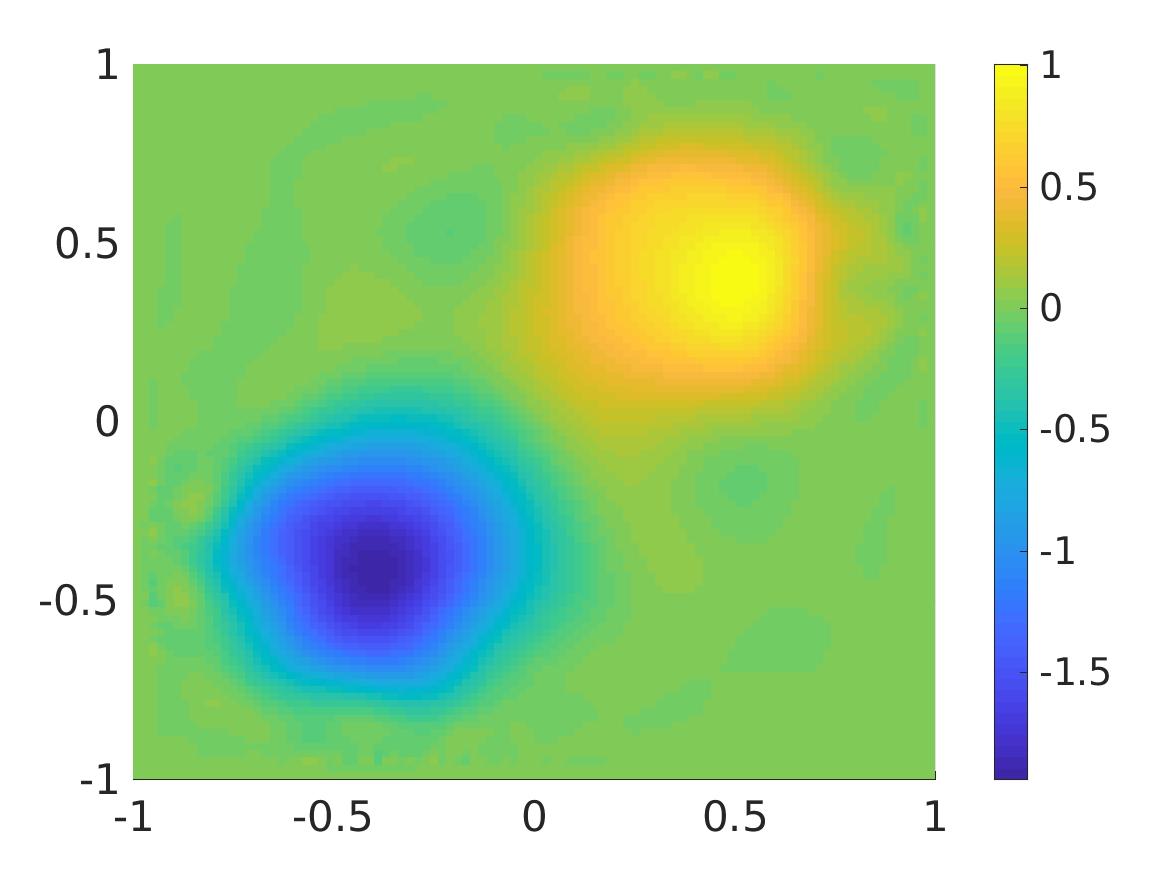}
		}
	\caption{\label{fig Test 2} Test 2. The true and computed source functions with different noise levels. Both positive and negative inclusions are sucessfully detected.}
	\end{center}
\end{figure}
The true and computed local extreme values of the source function at two inclusions are displayed in Table \ref{table test 2}. This table show that our method is stable with respect to noise.

\begin{table}[h!]
\caption{\label{table test 2} Test 2. The local extreme values of the functions $p_{\rm true}$ and $p_{\rm comp}$ at two inclusions. The relative error is denoted by ${\rm err}_{\rm rel}$. Inclusion 1 is the one on the top right and inclusion 2 is the one on the bottom left.}
\begin{center}
\begin{tabular}{|c|c|c|c|c|}
\hline
	Inclusion & noise level &extreme value$_{\rm true}$ & extreme value$_{\rm comp}$ & $({\rm err}_{\rm rel})$ 
	\\
	\hline
	1 & 0\%&1&0.988&1.2\%
	\\
	2&0\%&-2&-1.99&0.5\%
	\\
	1&5\%&1&0.984&1.6\%
	\\
	2&5\%&-2&-1.987&0.65\%
	\\
	1&10\%&1&1.004&0.4\%
	\\
	2&10\%&-2&-1.938&3.1\%
	\\
	\hline
\end{tabular}
\end{center}
\label{default}
\end{table}%

\noindent {\bf Test 3.} We next check the case when the source function is not smooth. 
In this case, we consider the piecewise constant function
	\begin{equation}
		p_{\rm true} = 
		\left\{
			\begin{array}{ll}
				-2 &\mbox{if } (x - 0.45)^2 + y^2 < 0.25^2\\
				2 &\mbox{if } 5(x + 0.45)^2 + \frac{1}{3} y^2 < 0.25^2.
			\end{array}
		\right.
	\label{ptrue test 3}
	\end{equation}
	The graph of the function $p_{\rm true}$ has two ``inclusions" with different shapes, a disk and an ellipse.
	The graphs of the true and computed source function are displayed in Figure \ref{fig Test 3}.	
	\begin{figure}[h!]
	\begin{center}
		\subfloat[The true source function]{
			\includegraphics[width = 0.45\textwidth]{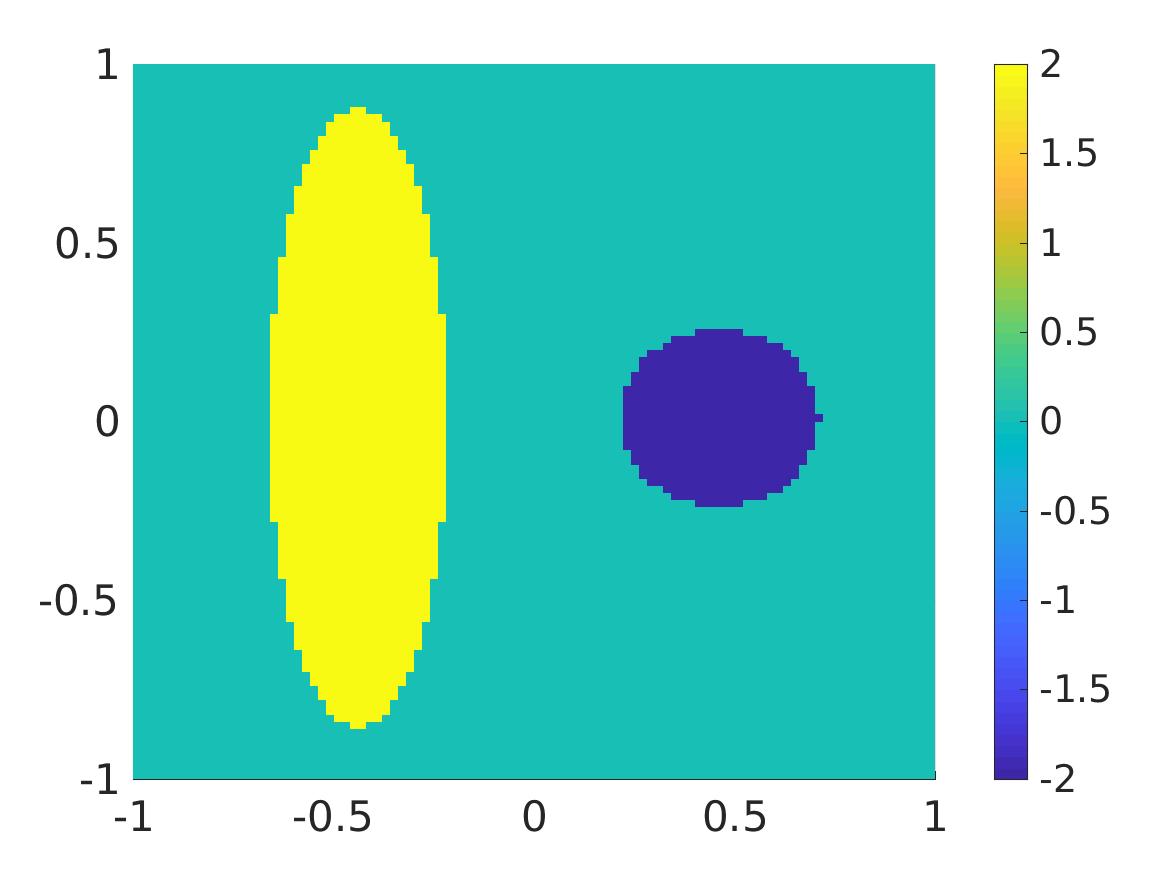}
		} \quad
		\subfloat[The computed source function, $\delta = 0\%$]{
			\includegraphics[width = 0.45\textwidth]{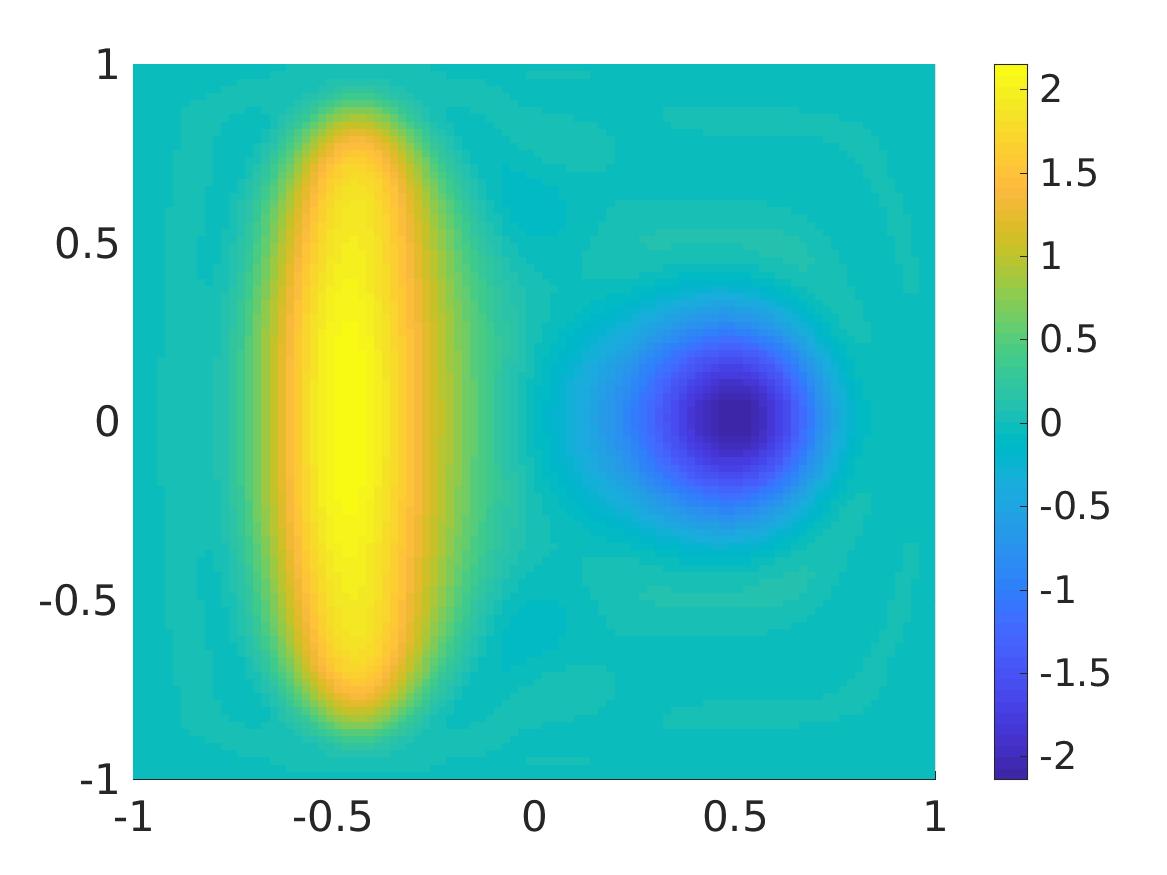}
		}
		
		\subfloat[The computed source function, $\delta = 5\%$]{
			\includegraphics[width = 0.45\textwidth]{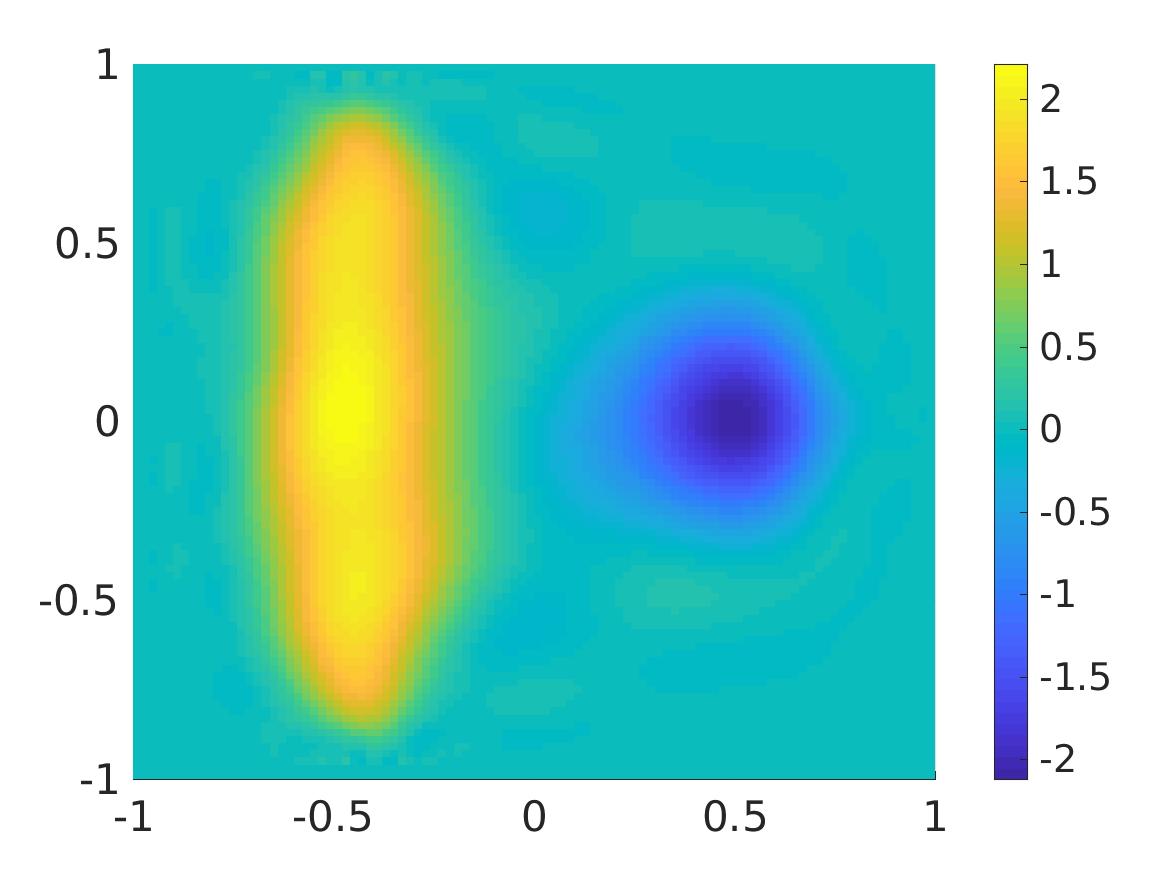}
		} \quad
		\subfloat[\label{fig 3d}The computed source function, $\delta = 10\%$]{
			\includegraphics[width = 0.45\textwidth]{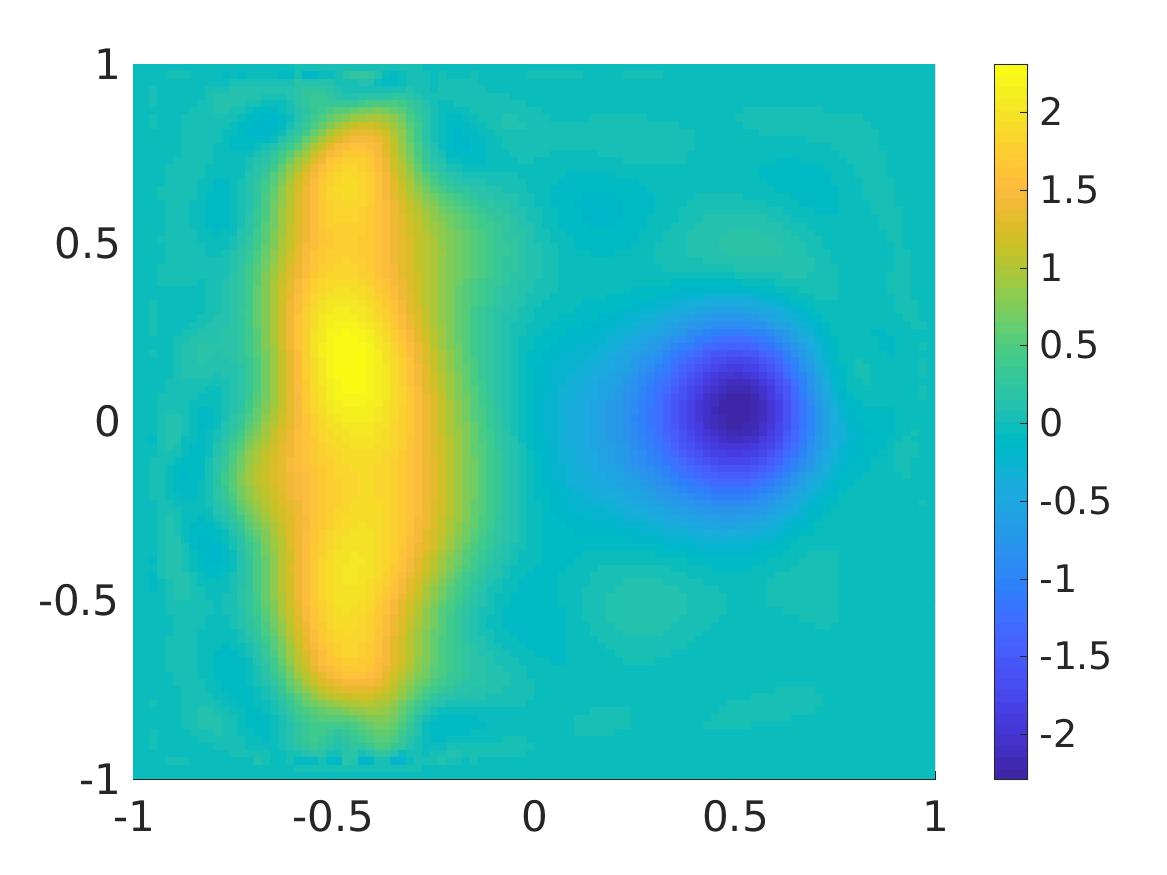}
		}
	\caption{\label{fig Test 3} Test 3. The true and computed source functions with different noise levels. Both ellipse and disk are successfully identified.}
	\end{center}
\end{figure}

The reconstruction of the image of the source function in this test is acceptable. 
Table \ref{table 2} shows the strength of our method in the sense that we can reconstruct the values of those two inclusions with acceptable error. 
\begin{table}[h!]
\caption{\label{table 2} Test 3. The true and reconstructed extreme values of inclusions. Inclusion 1 is the ellipse and inclusion 2 is the disk. The relative error is denoted by 
${\rm err}_{\rm rel}$. 
}
\begin{center}
\begin{tabular}{|c|c|c|c|c|}
\hline
Inclusion&Noise level&Extreme value$_{\rm true}$&Extreme value$_{\rm comp}$&${\rm err}_{\rm rel}$
\\
\hline
1&0\%&2&2.151&7.5\%
\\
2&0\%&-2&-2.133&6.7\%
\\
1&5\%&2&2.211&10.6\%
\\
2&5\%&-2&-2.117&5.9\%
\\
1&10\%&2&2.312&15.6\%
\\
2&10\%&-2&-2.29&14.5\%\%\\
\hline
\end{tabular}
\end{center}
\label{default}
\end{table}%

\noindent {\bf Test 4.} We test the nonsmooth true source function again with a complicated support. The true function is the characteristic function of the letter $\Omega.$ 
The numerical results for this test are shown in Figure \ref{fig test 4}.
\begin{figure}[h!]
	\begin{center}
		\subfloat[The true source function]{
			\includegraphics[width = 0.45\textwidth]{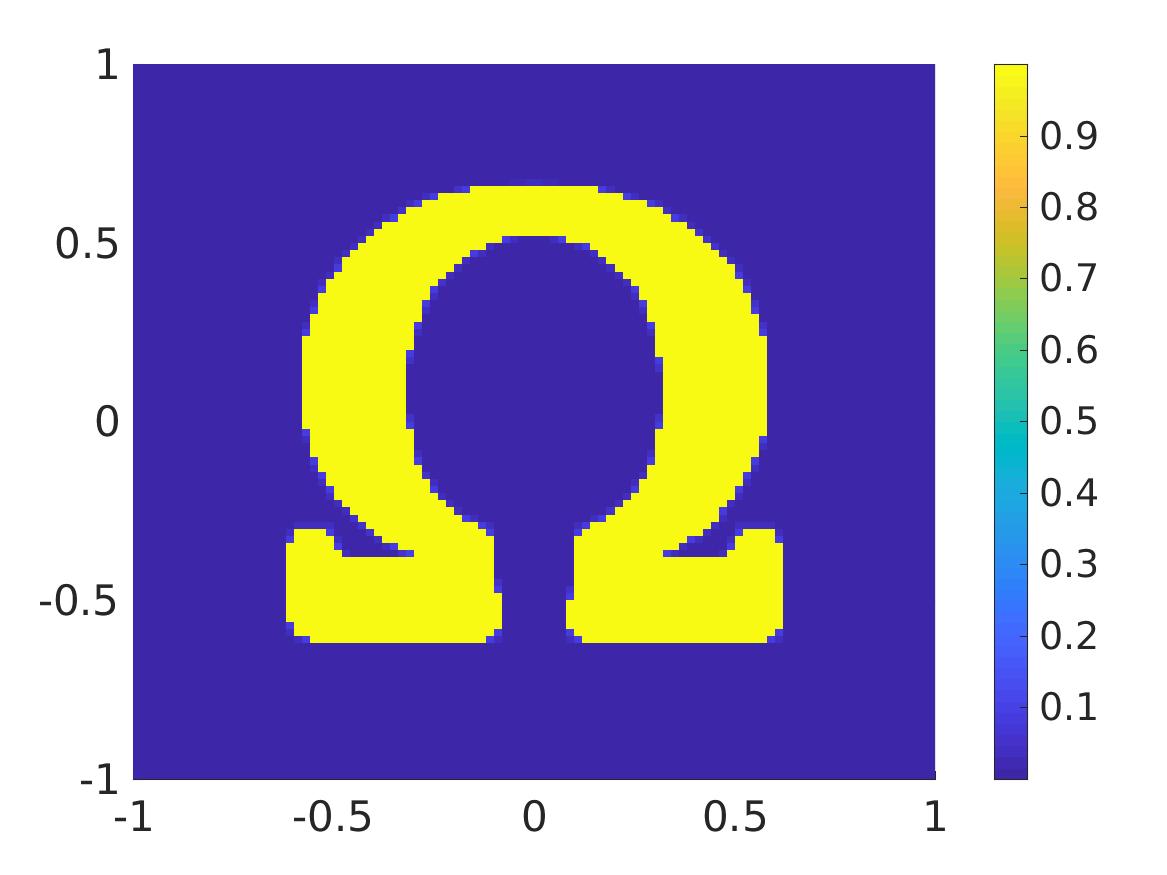}
		} \quad
		\subfloat[The computed source function, $\delta = 0\%$]{
			\includegraphics[width = 0.45\textwidth]{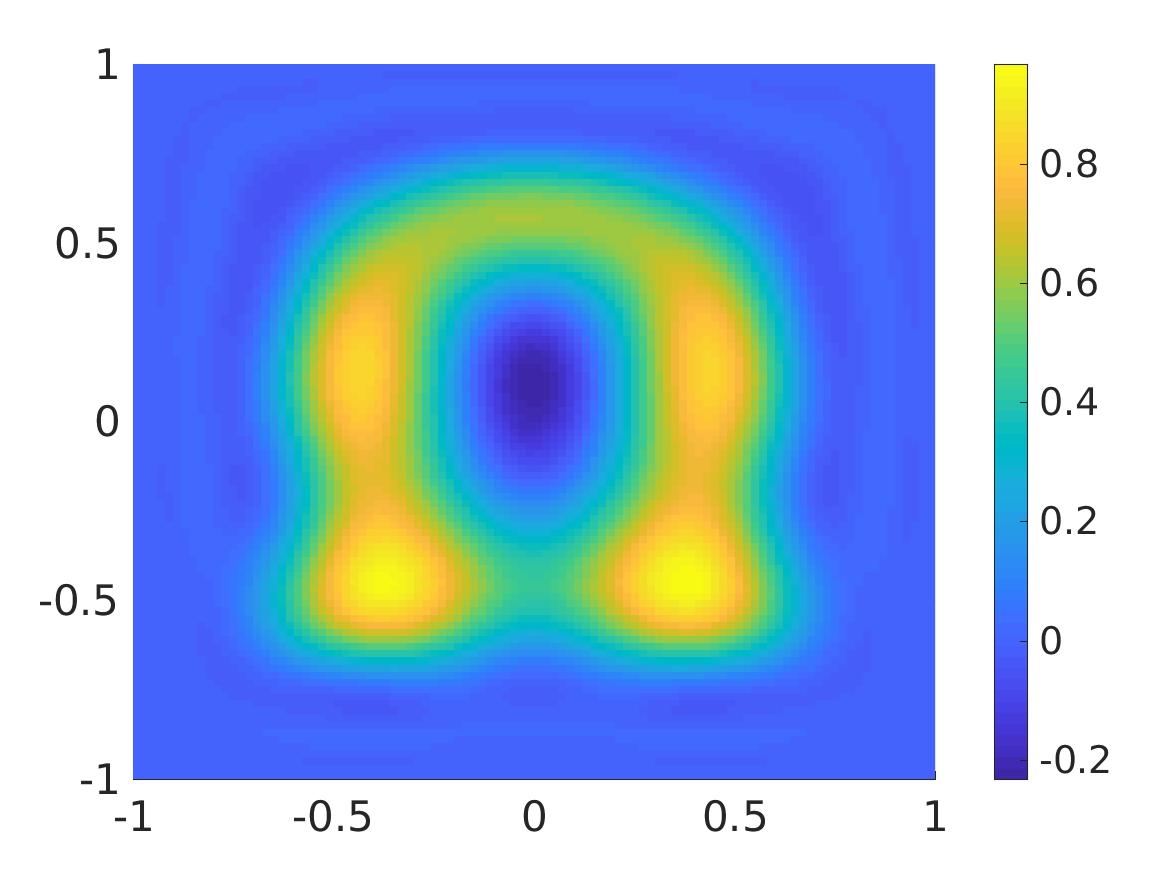}
		}
		
		\subfloat[The computed source function, $\delta = 5\%$]{
			\includegraphics[width = 0.45\textwidth]{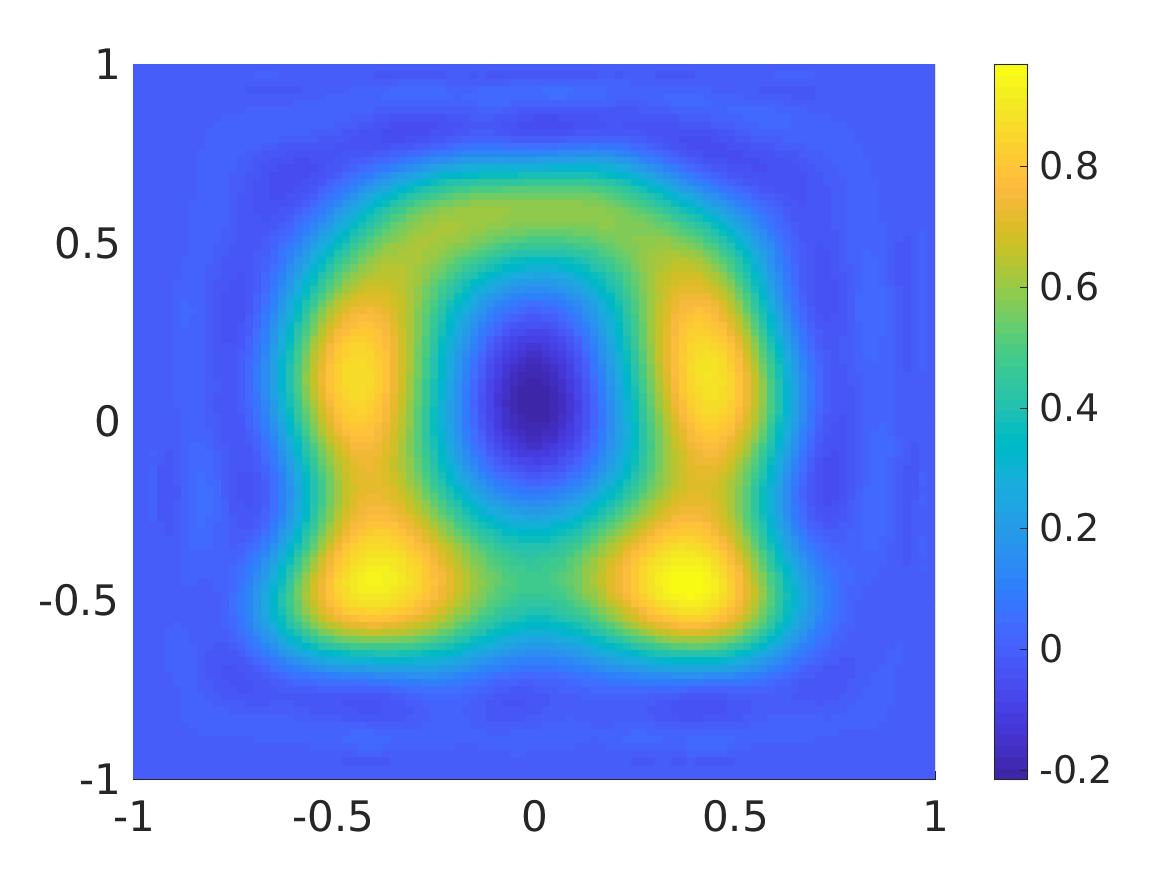}
		} \quad
		\subfloat[The computed source function, $\delta = 10\%$]{
			\includegraphics[width = 0.45\textwidth]{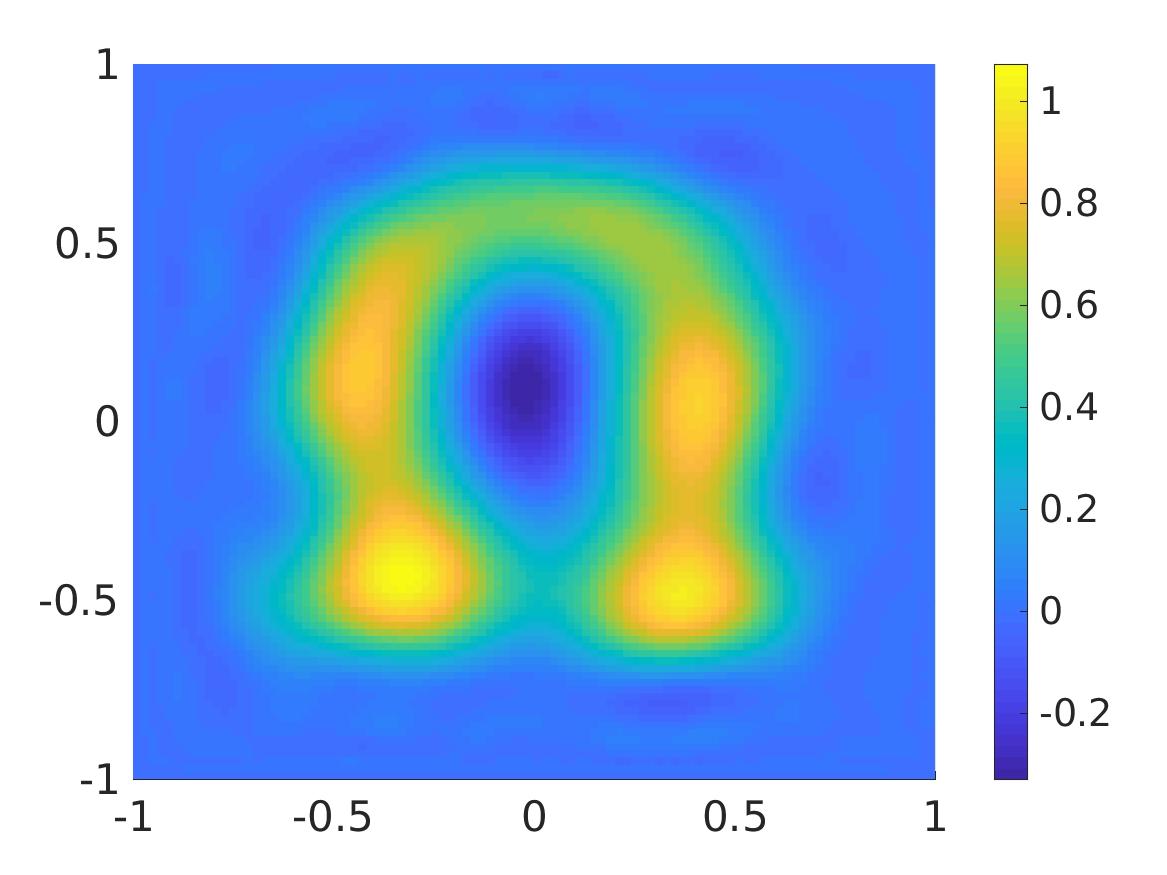}
		}
	\caption{\label{fig test 4} Test 4. The true and computed source functions with different noise levels. The letter $\Omega$ is successfully detected.}
	\end{center}
\end{figure}

Note that the maximal value of the reconstructed functions are acceptable. 
When $\delta = 0\%,$ $\max p_{\rm comp} = 0.97$ (relative error 3\%).
When $\delta = 5\%,$ $\max p_{\rm comp} = 0.97$ (relative error 3\%).
When $\delta = 10\%,$ $\max p_{\rm comp} = 1.073$ (relative error 7.3\%).

\begin{remark}
	Despite of the presence of the initial condition in equation \eqref{3.2}, it is evident that the quasi-reversibility method provides good numerical results with small relative errors. 
\end{remark}

\section{Application to a coefficient inverse problem} \label{sec cip}

In this section, we propose a numerical method to solve a severely ill-posed nonlinear coefficient inverse problem.

\subsection{The problem statement}
Problem \ref{ip} arises from a coefficient inverse problem for parabolic equations. 
For the simplicity, consider the problem of determining the coefficient $c(\bx)$ from the measurements of $\partial_n \mathfrak{u}(\bx, t)$ on $\partial \Omega \times [0, T]$ where, $\mathfrak{u}(\bx, t)$ is the solution of the following problem
\begin{equation}
	\left\{
		\begin{array}{rcll}
			\mathfrak{u}_t(\bx, t) &=&  \Delta \mathfrak{u}(\bx, t) + c(\bx) \mathfrak{u}(\bx, t)&\bx \in \Omega, t > 0\\
			\mathfrak{u}(\bx, t) &=& g_1(\bx, t) & \bx \in \partial\Omega, t > 0,\\			
			\mathfrak{u}(\bx, 0) &=& g(\bx) &\bx \in \Omega.
		\end{array}
	\right.
\label{2.4}
\end{equation}
Assume that the initial condition $g(\bx) > 0$ for all $\bx \in \Omega$ and the boundary condition $g_1$ satisfying $g_1(\bx, 0) = g(\bx)$ for all $\bx$ in $\partial \Omega.$ 
Consider the following nonlinear inverse problem.
\begin{problem}
	Let $T > 0.$ Determine the coefficient $c(\bx)$, $\bx \in \Omega$ from the measurement of
	\[
		F(\bx, t) = \partial_n \mathfrak{u}(\bx, t)
	\] for all $\bx \in \partial \Omega,$ $t \in [0, T].$
\label{pro coef}
\end{problem}	 

Problem \ref{pro coef} and its related versions are studied intensively.
Up to the knowledge of the author, the widely used method to solve this problem is the optimal control approach, see e.g., \cite{Borceaetal:ip2014, CaoLesnic:nmpde2018, CaoLesnic:amm2019, KeungZou:ip1998, YangYuDeng:amm2008} and references therein. 
The main drawback of this method is that the initial guess for the true solution is important to obtain numerical results. 
Un like this, we assume that we do not have any advanced knowledge of the true solution to Problem \ref{pro coef} and take the initial guess as a constant function.

Consider the circumstance that an initial guess for the function $c$, named as $c_0$, is known.
Then, we write
\begin{equation}
	c(\bx) = c_0(\bx) +  p(\bx).
	\label{perturbation}
\end{equation}
Denote by the function $\mathfrak{u}_0(\bx)$ the solution of \eqref{2.4} with $c_0$ replacing $c$ and let $w = \mathfrak{u} - \mathfrak{u}_0$.
It is not hard to see that
\begin{equation}
	\left\{
	\begin{array}{rcll}
		w_t(\bx, t) &=& \Delta w(\bx, t) + c_0 w(\bx, t) +  p(\bx)\mathfrak{u}(\bx, t) &\bx \in \Omega, t \in [0, T], \\
		w(\bx, t) &=& 0 &\bx \in \partial \Omega, t > 0,\\
		w(\bx, 0) &=& 0 &\bx \in \Omega.
	\end{array}
	\right.
	\label{1.44444}
\end{equation}
Since $c_0$ is an initial guess of $c$, we can replace the function $\mathfrak{u}$ in the differential equation in \eqref{1.44444} above by $ \mathfrak{u}_0$ to obtain
\begin{equation}
	\left\{
	\begin{array}{rcll}
		w_t(\bx, t) &=& \Delta w(\bx, t) + c_0 w(\bx, t) +  p(\bx)\mathfrak{u}_0(\bx, t) &\bx \in \Omega, t \in [0, T], \\
		w(\bx, t) &=& 0 &\bx \in \partial \Omega, t > 0,\\
		w(\bx, 0) &=& 0 &\bx \in \Omega
	\end{array}
	\right.
	\label{5.3}
\end{equation}
which leads to a particular case of Problem \ref{ip} with $f = \mathfrak{u}_0$.
We can compute $p(\bx)$ and therefore $c(\bx)$ via solving Problem \ref{ip} for the heat equation \eqref{5.3}.
Denoting the computed $c(\bx)$ by $c_1(\bx)$ and let $\mathfrak{u}_1(\bx, t)$ be the solution to \eqref{2.4} with $c = c_1$. 
We then find $c_2$ by solving Problem \ref{ip} for the heat equation \eqref{5.3} with $\mathfrak{u}_1$ replacing $\mathfrak{u}_0$.
The process is repeated to compute $c_3, c_4, \dots$ and we choose $c_{\rm comp} =c_{n^{*}}$ when $n^*$ is a fixed positive integer.
We summarize this numerical method to compute $c$ in Algorithm \ref{alg 2}.

\begin{remark}
	Imposing assumption \eqref{perturbation}, where the function $c_0$ is known and the unknown function $p$ is small, only plays the role of the suggestion for the ``linearization" analysis.
	However, 	in the reverse direction, the numerical results show that Algorithm \ref{alg 2} can be applied and provide good numerical results even in the case when $c_0$ is far away from the function $c$.
	Here, we understand by ``$c_0$ is far away from the function $c$" in two senses: (1) the complicated geometry of the true function $c$ and (2) the high contrasts.
\end{remark}

\begin{remark}		
The difficulty about the presence of an initial guess $c_0$ can be overcome  using the quasi-reversibility method. The authors of \cite{KlibanovNik:ra2017, Klibanov:ip2015} introduce a convex functional, which minimizer yields the solution of the problem under consideration, by combining the quasi-reversibility method and the Carleman weight functions. Numerical results in 1D are presented in \cite{KlibanovNik:ra2017}. 
It is important and interested to numerically test their method in higher dimensions.
\end{remark}

\begin{algorithm}[h!]
\caption{\label{alg 2}The procedure to solve Problem \ref{pro coef}}
	\begin{algorithmic}[1]
	\State\, Set $c_0$ as a background constant  and compute the solution  $\mathfrak{u}_0$  to \eqref{2.4} with $c_0$ replacing $c$. 
	\State\, Assume, by induction, that we know $c_n(\bx)$ and $u_n(\bx, t)$, $\bx \in \Omega$, $t \in [0, T]$. We find $c_{n+1}$ and $\mathfrak{u}_{n+1}$ as follows.
	\State\,  Compute the Neuman data $G_n(\bx, t) = F(\bx, t) - \partial_n \mathfrak{u}_n(\bx, t)$  for all $\bx \in \partial \Omega,$ $t \in [0, T].$
	\State\, Solve Problem \ref{ip} with $f(\bx, t) = \mathfrak{u}_n(\bx, t)$ and $G(\bx, t) = G_n(\bx, t)$ by Algorithm \ref{alg} to obtain a function $p_n(\bx)$.
		Set $c_{n+1}(\bx) = c_0 + p_n(\bx).$
	\State\, Choose $c_{\rm comp} = c_{n^*}$ where $n^*$ is chosen by numerical experiment. In this section, we set $n^* = 20.$
	\end{algorithmic}
\end{algorithm}

In the next subsection, we will show some numerical results.
We also display the graph of the relative difference 
\begin{equation}
	e_n = \frac{\|c_n - c_{n - 1}\|_{L^{\infty}(\Omega)}}{\|c_{n - 1}\|_{L^{\infty}(\Omega)}}, \quad n \geq 1
\end{equation}
to show the convergence of Algorithm \ref{alg 2}.

\subsection{Numerical results}
We perform two numerical results due to Algorithm \ref{alg 2} below. 
In these tests, the noise level is $5\%$. The background function $c_0(\bx) = 1$ for all $\bx \in \Omega.$

\noindent {\bf Test 5.}
The function $c_{\rm true}$ is the step function taking value 3 inside a letter $\Sigma$ and $1$ otherwise. 
We display the obtained numerical results in Figure \ref{fig Test 6}.
\begin{figure}[h!]
	\begin{center}
		\subfloat[The true coefficient $c_{\rm true}$]{\includegraphics[width=.3\textwidth]{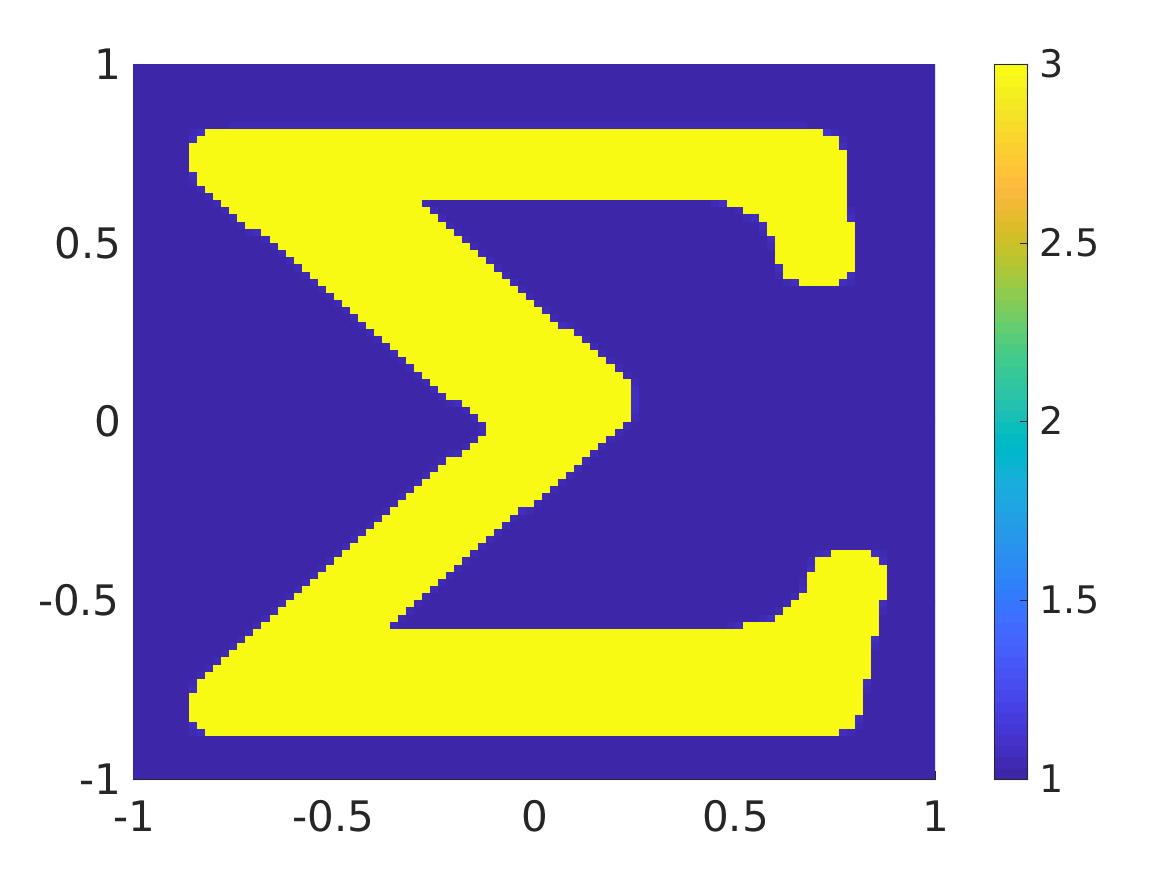}
		} \quad
		\subfloat[The function $c_1$]{\includegraphics[width=.3\textwidth]{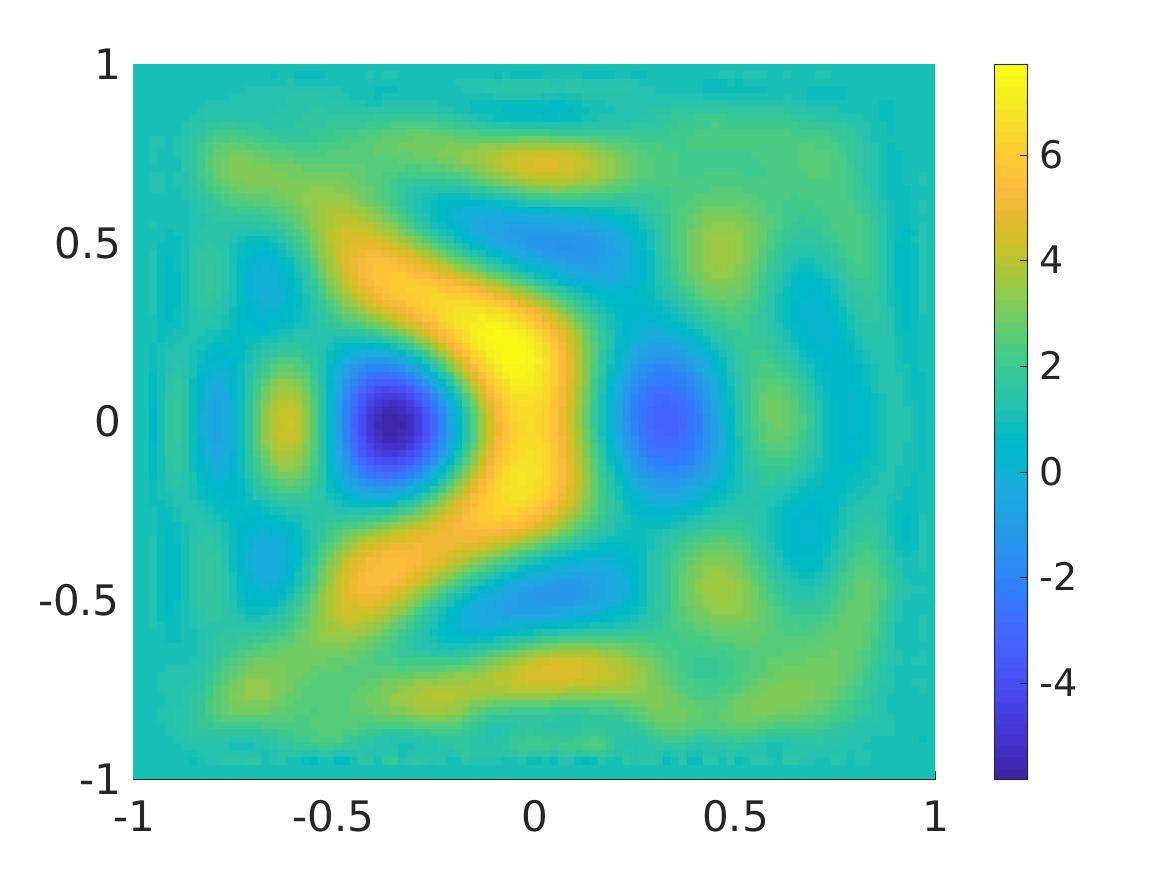}		
		} \quad
		\subfloat[The function $c_3$]{\includegraphics[width=.3\textwidth]{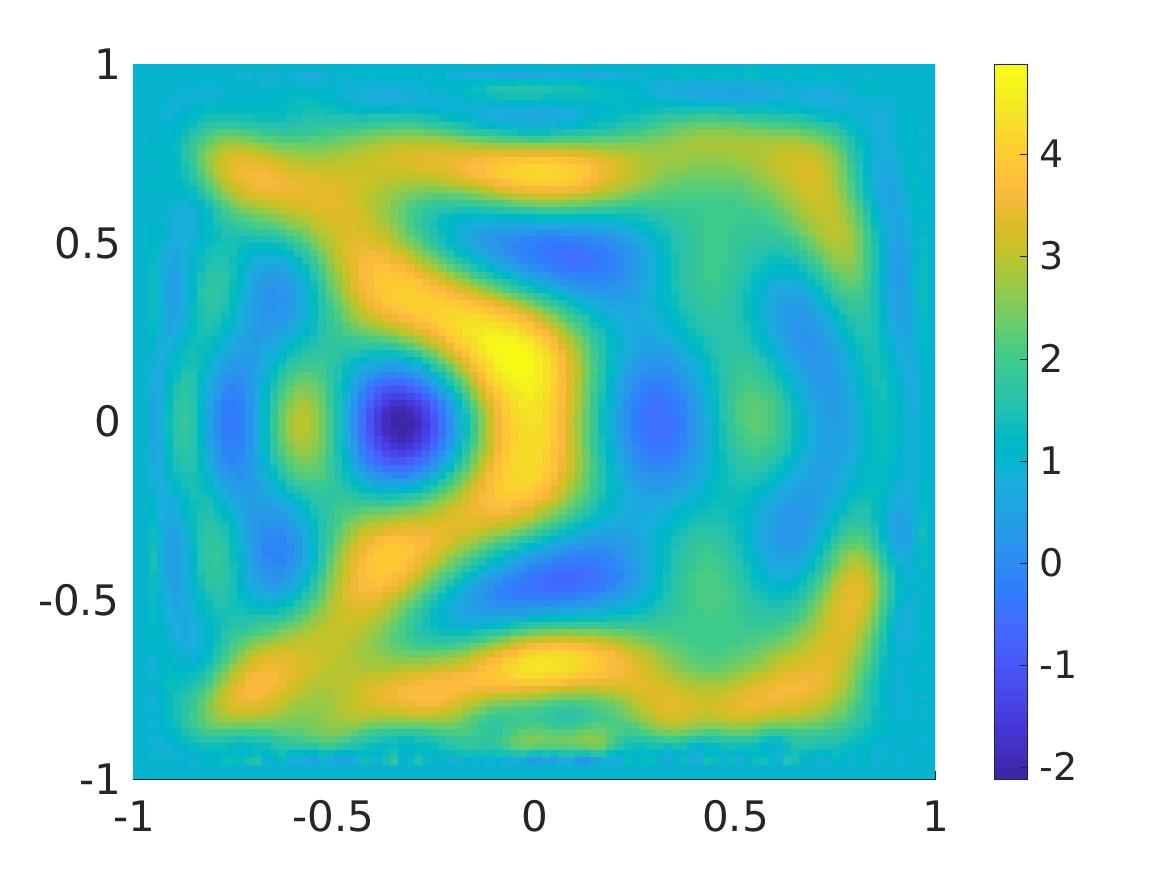}		
		}
		
		\subfloat[The function $c_{5}$]{\includegraphics[width=.3\textwidth]{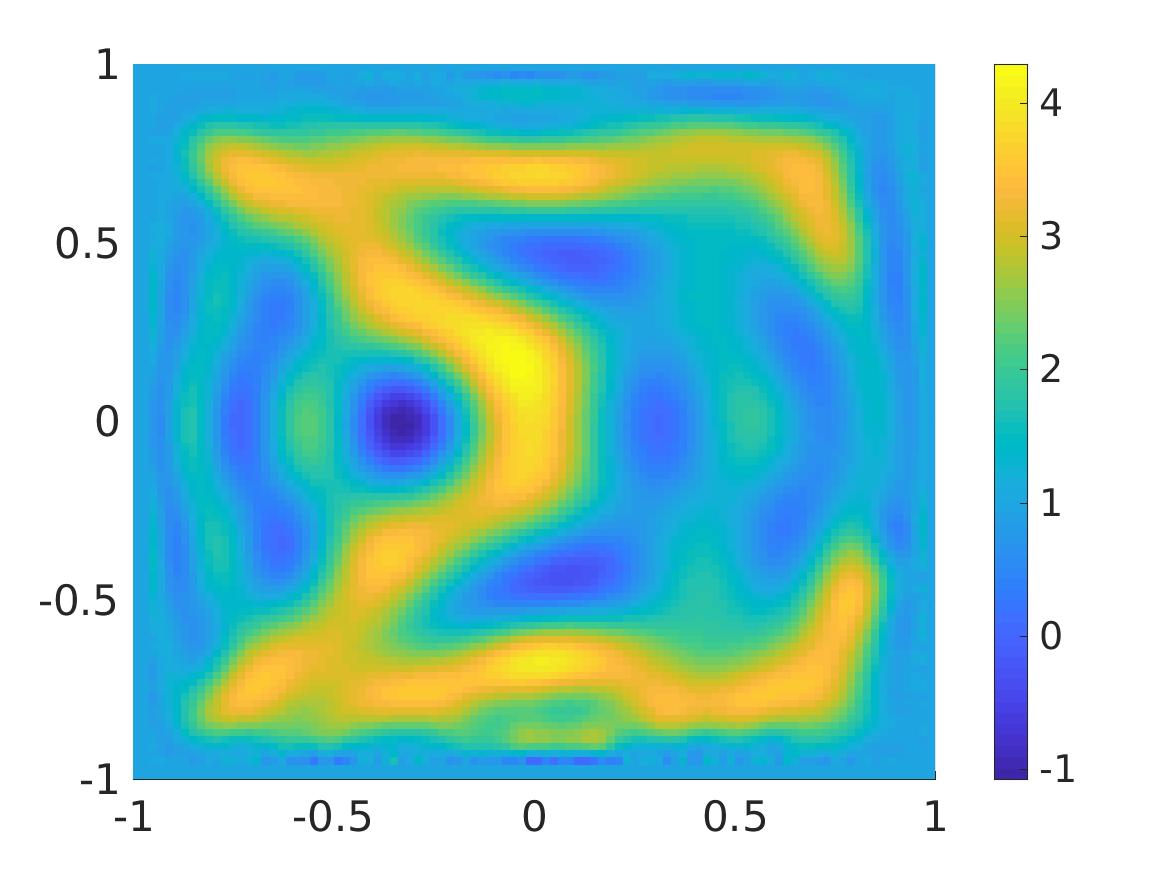}
		} \quad
		\subfloat[The function $c_{20}$]{\includegraphics[width=.3\textwidth]{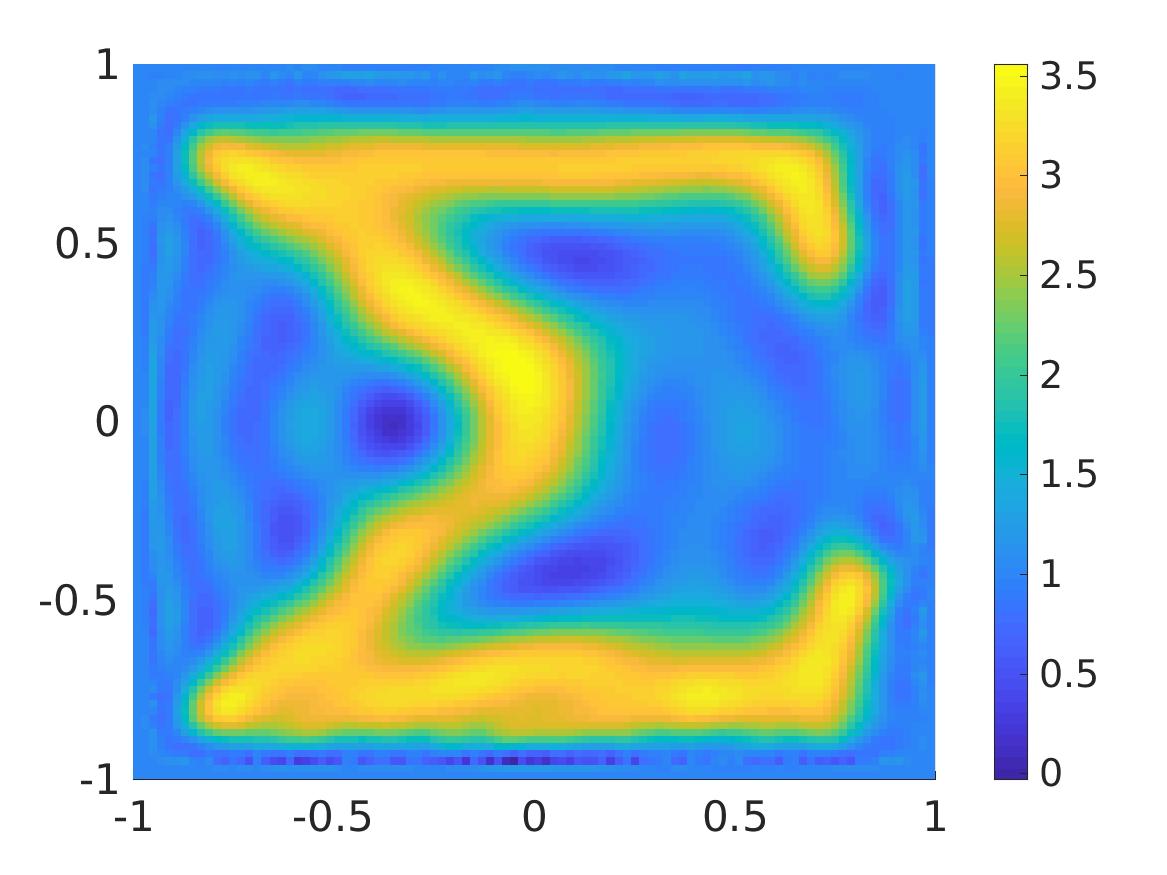}
		} \quad
		\subfloat[\label{error 5}The  relative difference $e_n$, $1 \leq n \leq 20$]{\includegraphics[width=.3\textwidth]{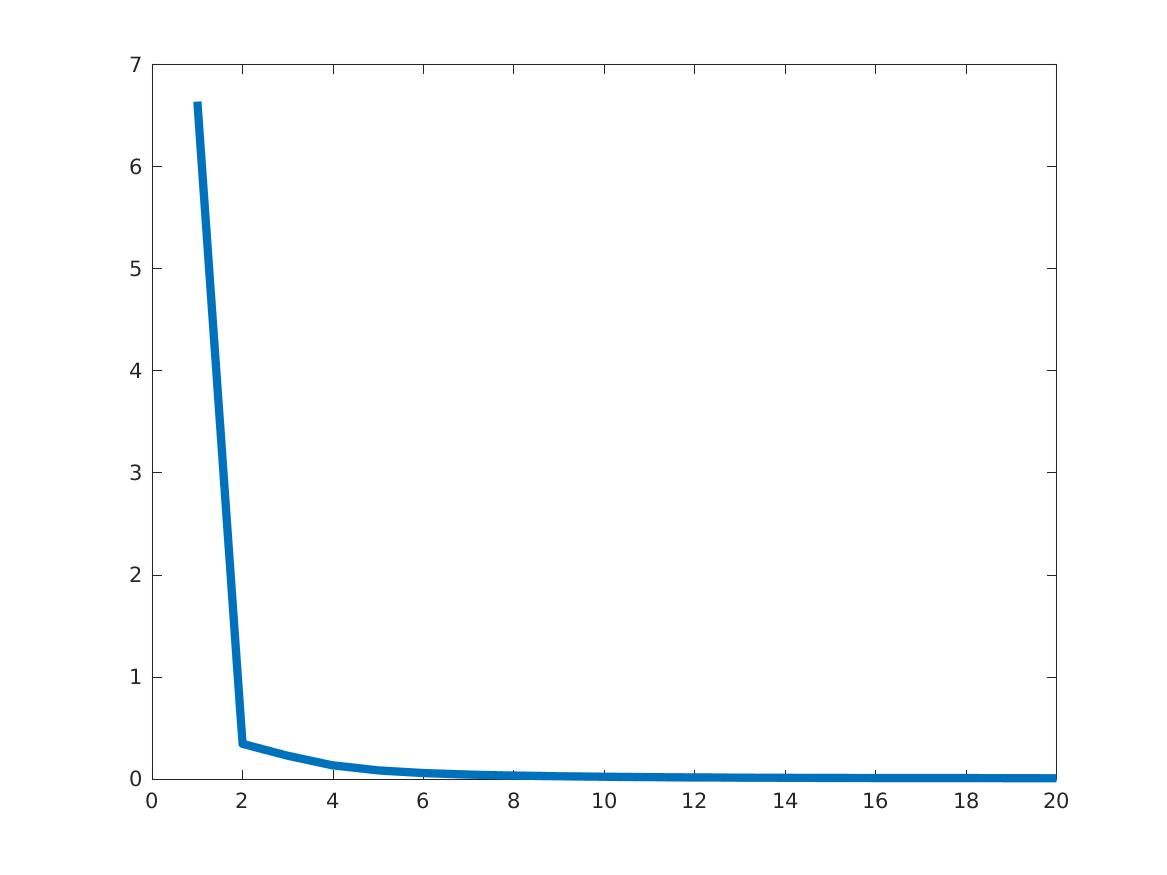} 
		}
		\caption{\label{fig Test 6} Test 5. Numerical solution to the coefficient inverse problem. The true and reconstructed coefficients $c$ and the recursive relative difference $e_n = \frac{\|c_n - c_{n - 1}\|_{L^{\infty}(\Omega)}}{\|c_{n - 1}\|_{L^{\infty}(\Omega)}}$, $1 \leq n \leq 20$.}
	\end{center}
\end{figure}
The reconstructed image ``$\Sigma$" meets the expectation although the initial guess $c_0 = 1$ is far away from the true function $c_{\rm true}$.  
The true maximal value of the function $c_{\rm true}$ is 3 and the reconstructed one is 3.56. The relative error is $18\%.$

\noindent {\bf Test 6.}
In test 6, the function $c_{\rm true}$ is given by
\[	
	c_{\rm true} = \left\{
		\begin{array}{ll}			
			5 & \max\{|x + 0.3|, 3|y| < 0.4\} \,\mbox{or}\, \max\{6|x - 0.5|, |y|\} < 0.8;\\		
			1 & \mbox{otherwise}.
		\end{array}
	\right.
\]
The image of the function $c_{\rm true}$ has a horizontal rectangle and a vertical rectangle. 
Due to the geometry and the high value, $c_{\rm true}$ is far away from the background $c_0 = 1$. 
We display the obtained numerical results in Figure \ref{fig Test 7}. 
Despite of the ``bad" initial guess $c_0$, the rectangles can be seen after a few iterations.
We note that the reconstructed images and value are improved with more iterations. The true maximum value of $c_{\rm true}$ is 5 and the reconstructed one is 5.94. The relative error is 18.8\%.

\begin{figure}[h!]
	\begin{center}
		\subfloat[The true coefficient $c_{\rm true}$]{\includegraphics[width=.3\textwidth]{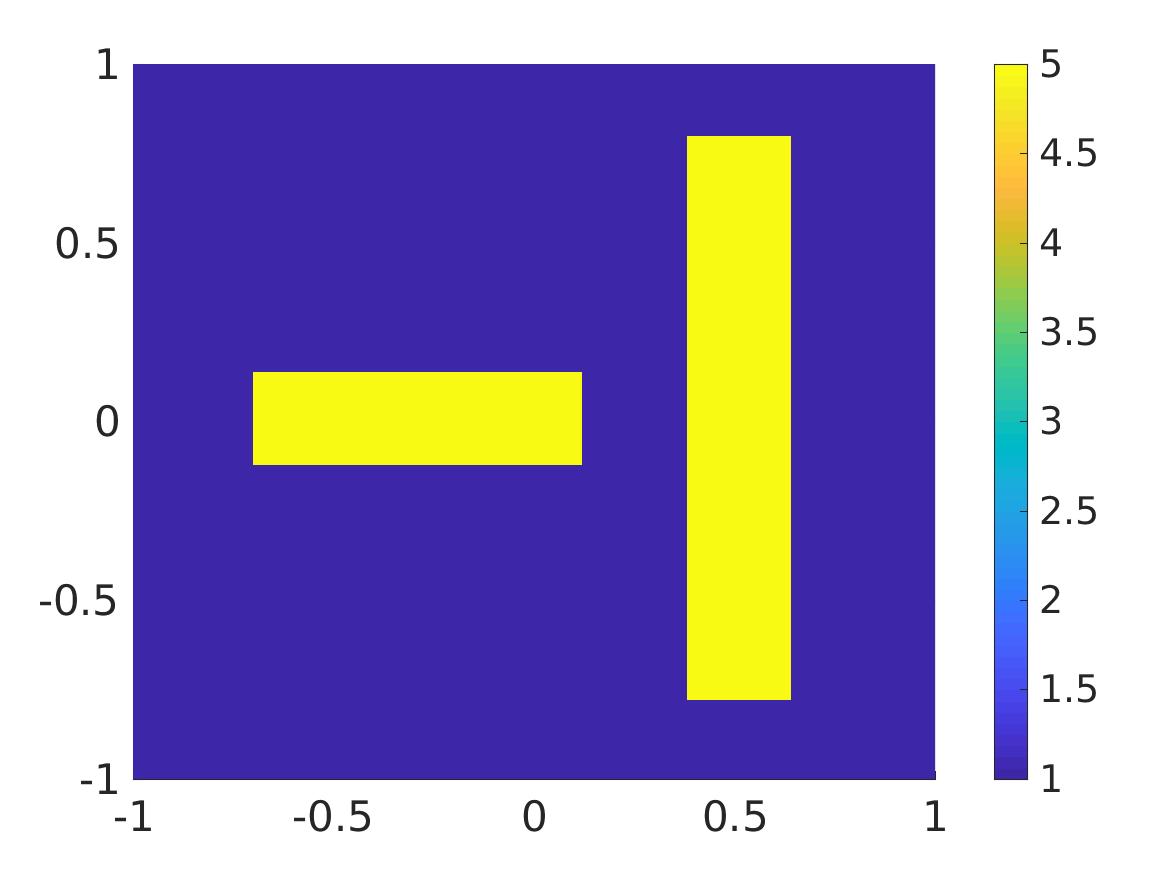}
		} \quad
		\subfloat[The function $c_1$]{\includegraphics[width=.3\textwidth]{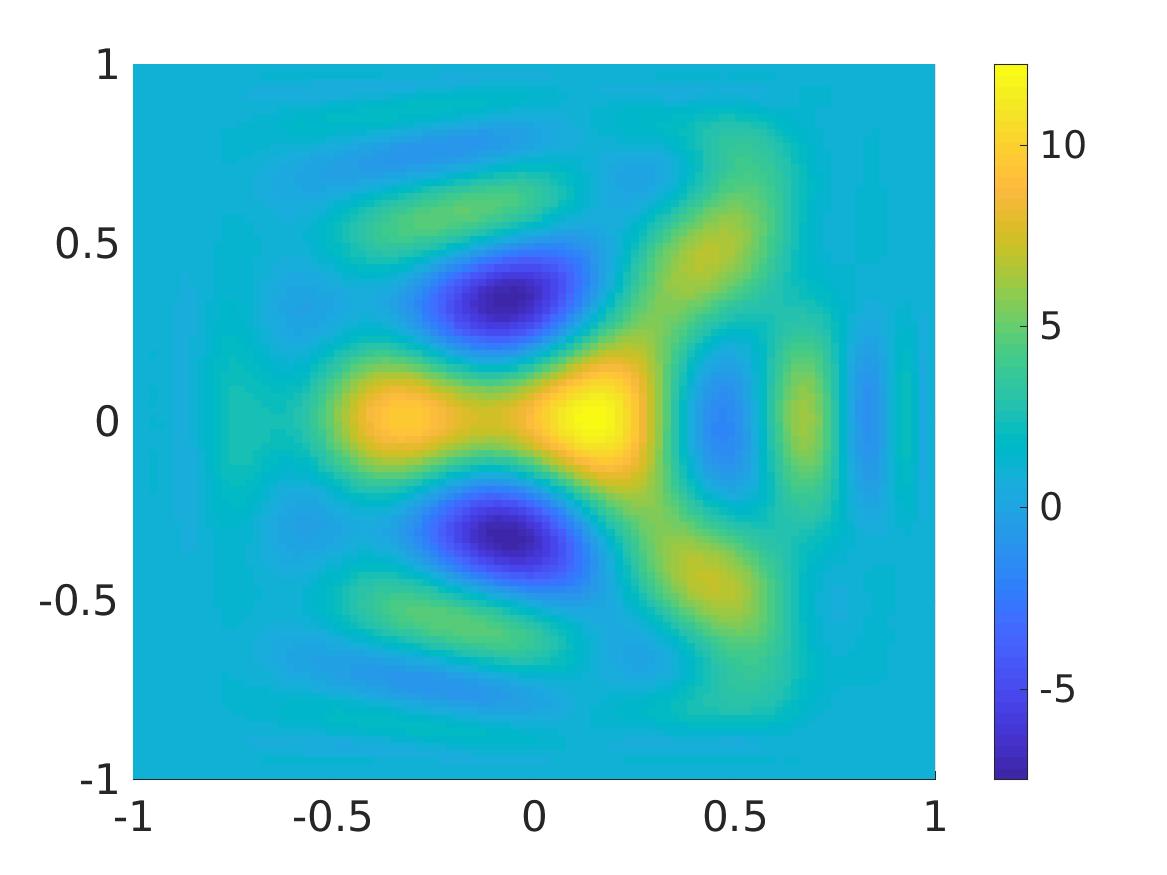}		
		} \quad
		\subfloat[The function $c_3$]{\includegraphics[width=.3\textwidth]{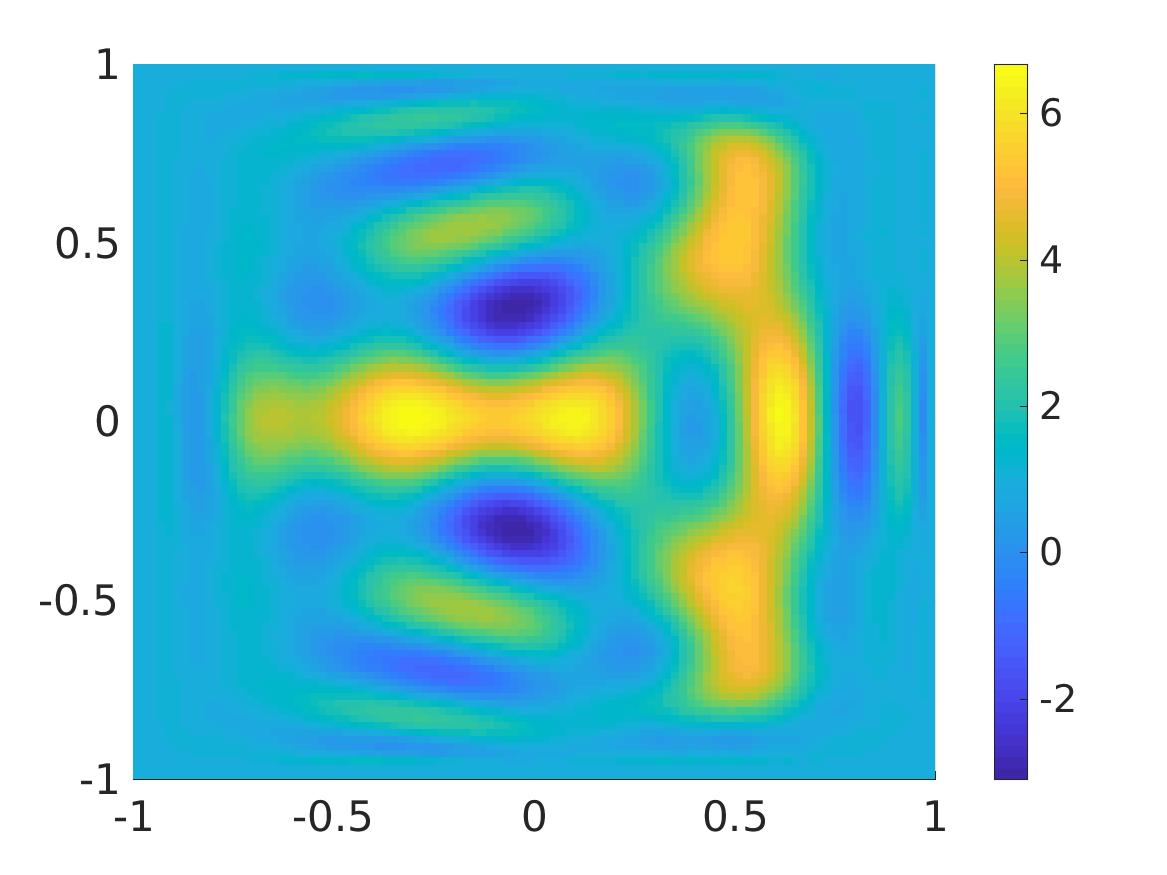}		
		}
		
		\subfloat[The function $c_{5}$]{\includegraphics[width=.3\textwidth]{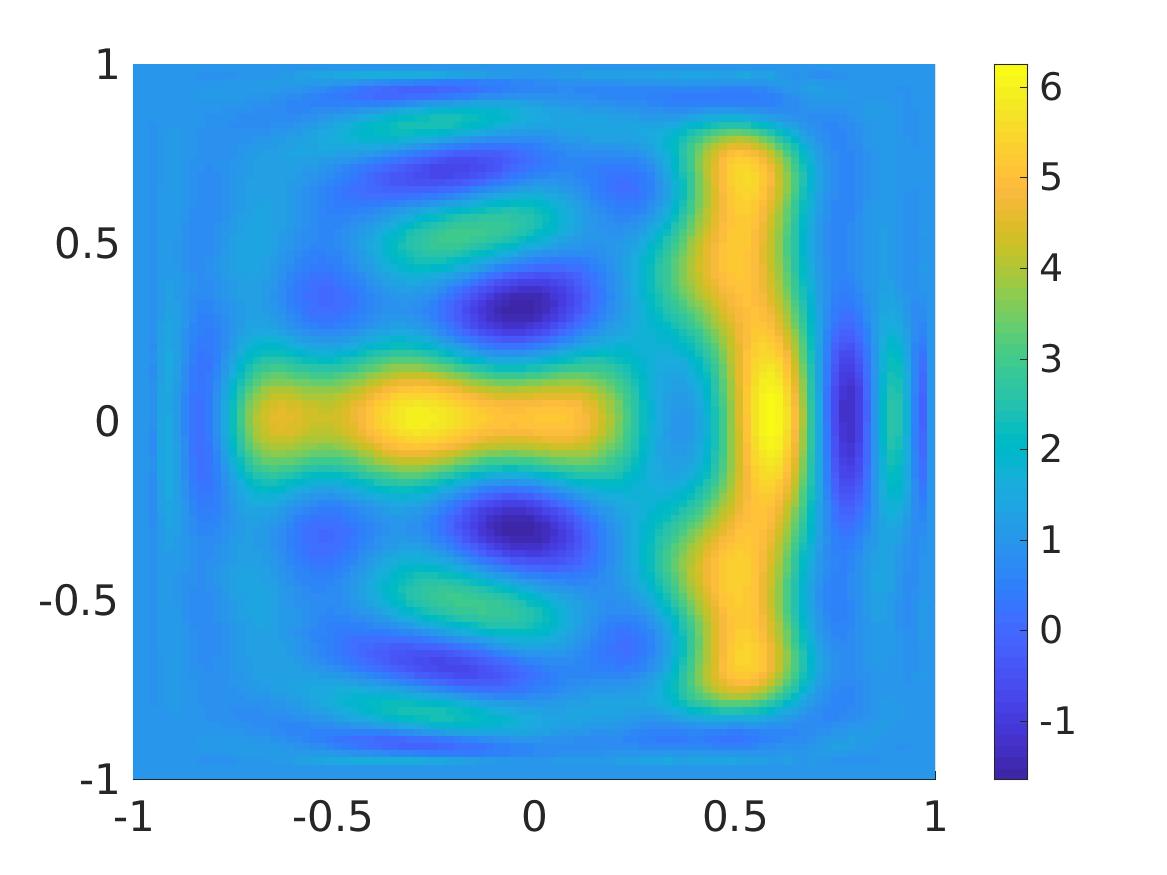}
		} \quad
		\subfloat[The function $c_{20}$]{\includegraphics[width=.3\textwidth]{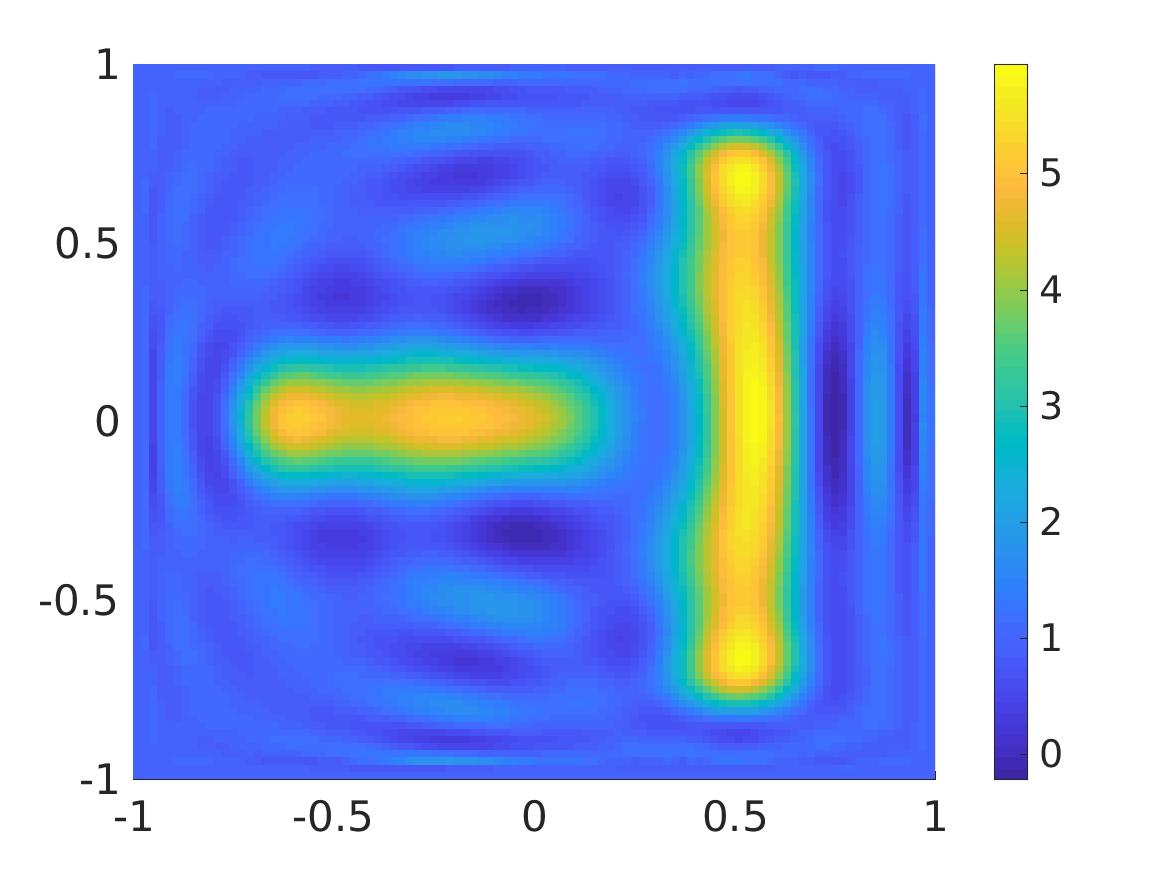}
		} \quad
		\subfloat[\label{error 7}The  relative difference $e_n$, $1 \leq n \leq 20$]{\includegraphics[width=.3\textwidth]{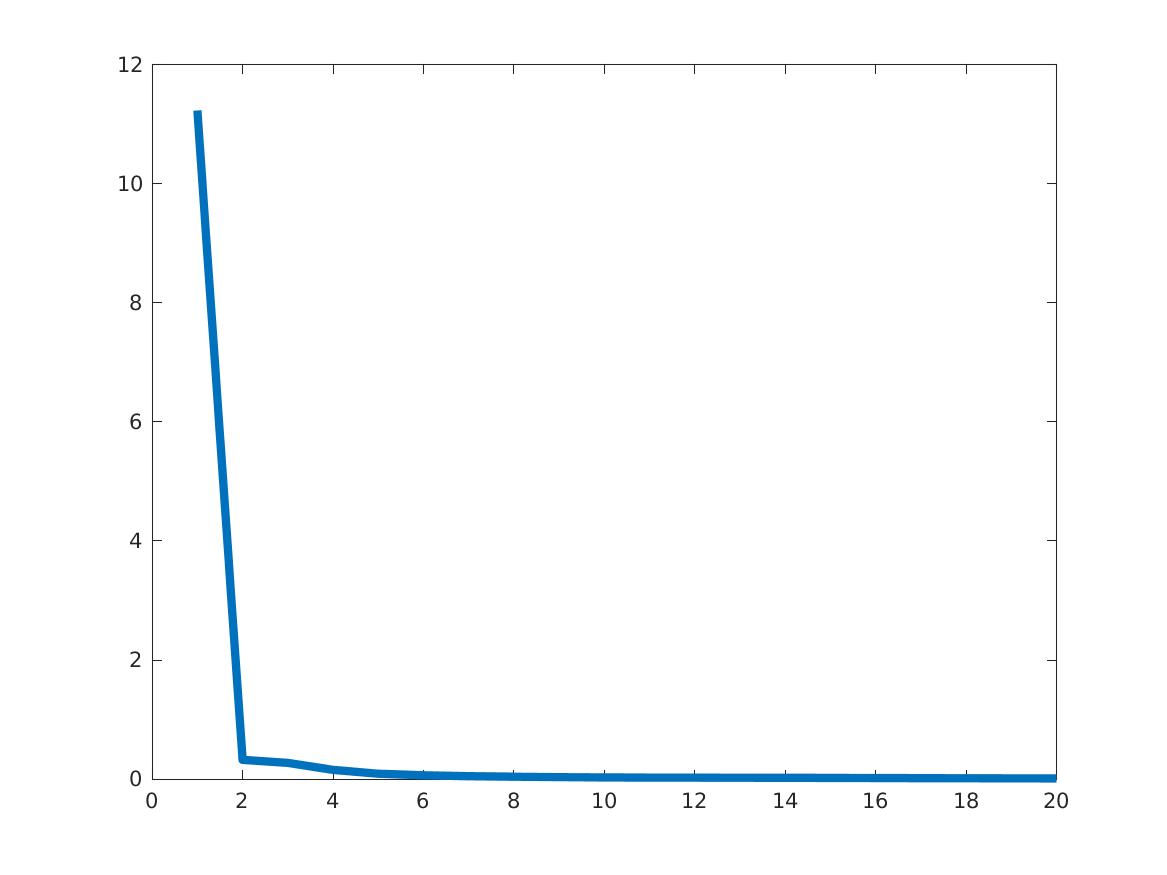} 
		}
		\caption{\label{fig Test 7}  Test 6. Numerical solution to the coefficient inverse problem. The true and reconstructed coefficients $c$ and the recursive relative difference $e_n = \frac{\|c_n - c_{n - 1}\|_{L^{\infty}(\Omega)}}{\|c_{n - 1}\|_{L^{\infty}(\Omega)}}$, $1 \leq n \leq 20$.}
	\end{center}
\end{figure}

\begin{remark}
	We observe that with the choice of $c_0$ as the background constant, the first reconstructed function $c_1$ is poor. 
	Then, in the next two iterations, the quality of the reconstructed function improves significantly.
 	Figures \ref{error 5} and \ref{error 7} show that the sequence $\{c_n\}_{n \geq 1}$ converges at the very fast rate.
\end{remark}

\section{Concluding remarks}\label{sec concluding}

In this paper, we have proposed a method to solve an inverse source problem for parabolic equations.
The stability of this problem is proved in an approximation context. 
To compute the numerical solutions to this inverse source problem, we derived an equation whose solution directly provides the desired solution of our inverse source problem.
However, this equation is not a standard parabolic equation. A theory to solve it is not yet available. 
We therefore employ the quasi-reversibility method to find its solution.
Since the inverse source problem in this paper is a linearization of a nonlinear coefficient inverse problem, we use the proposed method to establish an iterative method to solve that nonlinear coefficient inverse problem.
Numerical results were presented.

\section*{Acknowledgments}
The authors sincerely appreciate Michael V. Klibanov for many fruitful discussions.
The work of the second author was supported by US Army Research Laboratory and US Army Research Office grant W911NF-19-1-0044.

\providecommand{\bysame}{\leavevmode\hbox to3em{\hrulefill}\thinspace}
\providecommand{\MR}{\relax\ifhmode\unskip\space\fi MR }
\providecommand{\MRhref}[2]{%
  \href{http://www.ams.org/mathscinet-getitem?mr=#1}{#2}
}
\providecommand{\href}[2]{#2}


\end{document}